\documentclass[12pt,twoside]{article}
\title{Non-archimedean Monge-Amp{\`e}re measures for toric metrics on subvarieties of toric varieties}
\author{Chenying Lin\thanks{The author was supported by the SFB 1085 funded by DFG and the MICINN research project PID2023-147642NB-I00.}}

\usepackage[T1]{fontenc}

\usepackage[a4paper,hmargin=2.8cm,vmargin=2.0cm,includeheadfoot]{geometry}
\usepackage{float}
\usepackage[english]{babel}
\usepackage{amsmath}
\usepackage{amssymb}
\usepackage{dsfont}
\usepackage{array}
\usepackage{latexsym}
\usepackage{etoolbox}
\usepackage{mathtools}
\usepackage{underscore}

\usepackage{tocloft}

\usepackage{enumitem}

\usepackage{amsthm}
\theoremstyle{plain}

\usepackage{url}
\usepackage[dvipsnames]{xcolor}
\usepackage{graphicx}
\usepackage{tikz-cd}
\usepackage{multicol}
\usepackage{tikz}
\usetikzlibrary{arrows.meta,calc,positioning}

\usepackage{bm}
\usepackage{mathrsfs}  

\usepackage{aliascnt}

\newtheorem{thm}{Theorem}[section]

\newaliascnt{prop}{thm}
\newtheorem{prop}[prop]{Proposition}
\aliascntresetthe{prop}

\newaliascnt{cor}{thm}
\newtheorem{cor}[cor]{Corollary}
\aliascntresetthe{cor}

\newaliascnt{lem}{thm}
\newtheorem{lem}[lem]{Lemma}
\aliascntresetthe{lem}

\theoremstyle{definition}

\newaliascnt{defn}{thm}
\newtheorem{defn}[defn]{Definition}
\aliascntresetthe{defn}

\newaliascnt{rmk}{thm}
\newtheorem{rmk}[rmk]{Remark}
\aliascntresetthe{rmk}

\newaliascnt{ex}{thm}
\newtheorem{ex}[ex]{Example}
\aliascntresetthe{ex}

\usepackage{hyperref}
\hypersetup{colorlinks=true,citecolor=ForestGreen,linkcolor=blue,urlcolor=magenta}
\usepackage[capitalize,nameinlink,noabbrev]{cleveref}
\crefname{rmk}{Remark}{Remarks}
\crefname{thm}{theorem}{theorems}
\Crefname{thm}{Theorem}{Theorems}
\crefname{prop}{proposition}{propositions}
\Crefname{prop}{Proposition}{Propositions}
\crefname{cor}{corollary}{corollaries}
\Crefname{cor}{Corollary}{Corollaries}
\crefname{lem}{lemma}{lemmas}
\Crefname{lem}{Lemma}{Lemmas}
\crefname{ex}{example}{examples}
\Crefname{ex}{Example}{Examples}

\newcommand{\R}{\mathbb{R}} 
\newcommand{\Z}{\mathbb{Z}} 
\newcommand{\N}{\mathbb{N}} 
\newcommand{\Q}{\mathbb{Q}} 
\newcommand{\C}{\mathbb{C}}

\newcommand\an{\operatorname{an}}

\newcommand\trop{\operatorname{trop}}
\newcommand{\ddc}{\mathrm d'\mathrm d''}

\tikzset{
  outer edge/.style={draw=black, line width=0.9pt},
  inner edge/.style={draw=black, line width=0.85pt},
  highlight edge/.style={draw=blue, line width=1.4pt},
  final highlight/.style={draw=black, line width=1.15pt},
  flow arrow/.style={draw=black, line width=0.9pt, -{Latex[length=3mm,width=2mm]}},
  blue text/.style={text=blue, font=\Large},
  title text/.style={font=\Large}
}

\newtheorem{introthm}{Theorem}

\crefname{introthm}{theorem}{theorems}
\Crefname{introthm}{Theorem}{Theorems}

\newtheorem{introcor}{Corollary}

\crefname{introcor}{corollary}{corollaries}
\Crefname{introcor}{Corollary}{Corollaries}

\begin{document}

\maketitle

\begin{abstract}
    We study non-archimedean Monge-Amp{\`e}re measures of toric metrics on closed subvarieties of toric varieties, after reducing, if necessary, to the minimal torus-orbit closure containing the subvariety. The measure vanishes on the toric boundary and is supported on the tropical skeleton. Our main result explicitly describes their local structure in terms of the tropical germ of the subvariety and the tropical Green function of the metric. More precisely, the tropicalized Monge-Amp{\`e}re measure is the corresponding polyhedral Monge-Amp{\`e}re measure, which extends the tropical intersection product, while the original measure is recovered locally by tropical pullback from explicit tropical and convex data. This provides a toric counterpart to analogous results for subvarieties of abelian varieties, with new local phenomena arising from tropicalization. 

\end{abstract}

\vspace{1.5em}

\setcounter{tocdepth}{1}

\tableofcontents

\newpage

\section{Introduction}

\paragraph{Background and motivation.}
Monge-Amp{\`e}re measures are central objects in both archimedean and non-archimedean Arakelov geometry. In the archimedean setting, Yau's solution of the Calabi conjecture shows that, within a fixed K{\"a}hler class, a K{\"a}hler metric is uniquely determined by its Monge-Amp{\`e}re measure \cite{Yau}; the study of complex Monge-Amp{\`e}re measures has since become a central theme in pluripotential theory. On the non-archimedean side, Chambert-Loir \cite{ChambertLoir} associated to semipositive metrized line bundles $(L_1,\lVert\cdot\rVert_1),\ldots,(L_d,\lVert\cdot\rVert_d)$ on a $d$-dimensional proper variety $X$ a positive Radon measure
\[
c_1(L_1,\lVert\cdot\rVert_1)\wedge\cdots\wedge c_1(L_d,\lVert\cdot\rVert_d)
\]
on the Berkovich analytification $X^{\an}$. In particular, for a semipositive metrized line bundle $(L,\lVert\cdot\rVert)$, we refer to $c_1(L,\lVert\cdot\rVert)^{\wedge d}$ as its non-archimedean Monge-Amp{\`e}re measure.

These measures provide the local measure-theoretic counterpart of arithmetic intersection products. They also arise as limiting measures in equidistribution theorems \cite{YuanEquidistibution, Gubequidistri}. Such equidistribution results have important applications to Bogomolov-type problems, notably in the work of Gubler and Yamaki and in the subsequent proof of the geometric Bogomolov conjecture by Xie and Yuan \cite{GubBogomolov,yamaki2016strict,yamaki2018trace,XieYuan}. Thus, non-archimedean Monge-Amp{\`e}re measures form a bridge between local intersection theory, non-archimedean analytic geometry, and applications in Diophantine geometry. These methods have also found applications in mirror symmetry, for instance, in Yang Li's work relating non-archimedean Monge-Amp{\`e}re equations to the metric SYZ conjecture \cite{SYZ1,SYZ2}.

\paragraph{The toric setting.}
The purpose of this paper is to give a structural description of non-archimedean Monge-Amp{\`e}re measures for toric metrics on closed subvarieties of toric varieties. By translating these measures into tropical and polyhedral data, our description also makes them more explicit and easier to compute. 

The toric setting considered here is natural for several reasons. Closed subvarieties are basic objects in intersection theory and Arakelov geometry. Moreover, the class of closed subvarieties of toric varieties is broad: in particular, every projective variety belongs to this class, since it admits a closed embedding into a projective space. At the same time, the toric ambient space provides a natural tropicalization, which makes the geometry of the subvariety along the dense torus visible as a weighted tropical cycle. 

The theory of toric metrics, systematically developed by Burgos Gil, Philippon, and Sombra \cite{BPS}, provides a parallel description on the metric side. A toric metric on a toric line bundle is encoded by its tropical Green function on the real cocharacter space $N_\R$, and semipositivity translates into a convexity condition on this function. Semipositive toric metrics form a natural class that includes canonical metrics.

Our results provide a toric counterpart to Gubler's structural description of canonical measures on subvarieties of abelian varieties \cite{GubCanonical} and to its extension by Gubler and Stadlöder to Monge-Amp{\`e}re measures associated with toric metrics on subvarieties of abelian varieties \cite{GSMongeAmpere}. 

\paragraph{Tropical skeletons.}
The tropical descriptions of both the subvariety and the metric suggest that the Monge-Amp{\`e}re measure should be closely related to a natural polyhedral subset of the Berkovich analytification associated with tropicalization.

Let $X$ be a $d$-dimensional closed subvariety of a proper toric variety $Y$ with dense torus $T$ over a complete algebraically closed field $K$ with a non-trivial non-archimedean absolute value, and let $L$ be a toric line bundle on $Y$ with a semipositive toric metric $\lVert\cdot\rVert$. We may assume that $X\cap T\neq\varnothing$, after replacing $Y$, if necessary, by the closure of a suitable toric orbit. We denote by
\[
\trop:X^{\an}\cap T^{\an}\to N_\R
\]
the tropicalization map. 

The Monge-Amp{\`e}re measure $c_1(L|_X,\lVert\cdot\rVert)^{\wedge d}$ vanishes on the toric boundary $X^{\an}\backslash T^{\an}$, and its restriction to $X^{\an}\cap T^{\an}$ is supported on the tropical skeleton $\Sigma$; see \Cref{prop:MA support}. 

Skeletons provide a bridge between non-archimedean analytic geometry and polyhedral geometry. Skeletons arising from strictly semistable formal models have played an important role in the study of non-archimedean measures, notably in Gubler's description of canonical measures on subvarieties of abelian varieties; see \cite[Section 4]{BerkovichS1}, \cite[Section 4]{BerkovichS2} and \cite[Section 5]{GubCanonical}. However, in our setting, the natural object is the tropical skeleton, introduced by Ducros \cite{Duc12}, associated directly with the fixed tropicalization map, whose polyhedral structure reflects the local tropical geometry of the subvariety.

\paragraph{Polyhedral Monge-Amp{\`e}re measures.}
Let $\Gamma$ be the value group of $K$. To formulate our main theorem, we use a $\Gamma$-rational version of the polyhedral Monge-Amp{\`e}re measure recently introduced by Botero, Mazzon, and Pille-Schneider \cite{MApoly}. Their construction is a Monge-Amp{\`e}re operator on balanced $(\Z,\Q)$-polyhedral spaces, defined for rational piecewise affine functions by means of tropical intersection theory \cite{AllermannRau} and extended, by approximation, to a suitable class of continuous convex functions. On the maximal-dimensional open faces, it agrees with a multiple of the usual real Monge-Amp{\`e}re measure in the sense of Alexandrov \cite{realMA}. In contrast to the latter, however, the polyhedral Monge-Amp{\`e}re measure is defined intrinsically on the entire polyhedral space and thus also incorporates its lower-dimensional polyhedral structure. 

The rationality assumptions in \cite{MApoly} are not sufficient for our purposes. Since the valuation is not assumed to be discrete in our setting, the polyhedral complexes arising from tropicalization are $(\Z,\Gamma)$-polyhedral complexes, while the piecewise affine functions appearing in our approximation arguments are naturally $\Gamma$-rational, meaning that their affine pieces have rational slopes and constant terms in $\Gamma$. 
In \Cref{Section: poly MA}, we therefore develop the version of relevant polyhedral Monge-Amp{\`e}re results needed here for $(\Z,\R)$-polyhedral complexes and piecewise affine functions with rational slopes. 
The key idea is that, after translation, a $(\Z,\R)$-polyhedral complex is locally identified with a rational fan, so the locality of the polyhedral Monge-Amp{\`e}re measure in \cite[Corollary 4.18]{MApoly} allows us to reduce to the rational setting.

For the class of convex functions obtained as uniform limits of the $\Gamma$-rational piecewise affine functions considered here, we further identify the polyhedral Monge-Ampère measure with the Bedford–Taylor product of \cite[Section 4.1]{BGJK21} from tropical pluripotential theory; see \Cref{prop: PS comparison of poly MA and delta form} and \Cref{prop: comparison of poly MA and BT product}

This construction also provides a reason for working on the tropical side. The framework of \cite{MApoly} applies to balanced $(\Z,\Q)$-polyhedral spaces embedded in an ambient real vector space $N_\R$, but the tropical skeleton is not naturally of this form. We therefore describe the non-archimedean Monge-Amp{\`e}re measures after tropical pushforward.

In our setting, let $g_{\trop}:N_\R\to\R$ be the tropical Green function associated with the metric $\lVert\cdot\rVert$. Then $g_{\trop}$ restricts to a function on a weighted representative $C_x$ of the tropical variety $\trop(X^{\an}\cap T^{\an},x)$ of the germ $(X^{\an}\cap T^{\an},x)$ at $x\in\Sigma$, and we denote the resulting polyhedral Monge-Amp{\`e}re measure by
\[
\operatorname{MA}_{\operatorname{poly},C_x}(g_{\trop}).
\]

\paragraph{Main results.}
We are now ready to state our main results. Together with \Cref{prop:MA support}, they give a complete description of the local structure of the non-archimedean Monge-Amp{\`e}re measure for toric metrics on closed subvarieties of toric varieties in terms of the germ and the tropical Green function of the toric metric. This also makes the measure accessible for explicit computation on the tropical side.

Recall that $\Sigma$ is the tropical skeleton of $X^{\an}\cap T^{\an}$ with respect to the map $\trop$. For simplicity, in the following, a local weighted polyhedral chart of $\Sigma$ will always mean one compatible with $\trop$; see \Cref{Section: tropical skeletons} for the precise definition.

\begin{introthm}\label{thm:intro main}
    Let $x\in\Sigma$. Then there exists an open neighborhood $U(x)$ of $x$ and a weighted representative $C_x$ of the germ $\trop(X^{\an}\cap T^{\an},x)$ with support $\trop(U(x))$ such that the following holds.

    For every toric line bundle $L$ on $Y$ endowed with a semipositive toric metric $\lVert\cdot\rVert$ with tropical Green function $g_{\trop}:N_\R\to\R$, we have
    \[
    \left(\trop|_{U(x)}\right)_*\left(c_1(L|_X,\lVert\cdot\rVert)^{\wedge d}|_{U(x)}\right)=\operatorname{MA}_{\operatorname{poly},C_x}\left(g_{\trop}\right).
    \]

    Moreover, let $\Delta$ be an open face containing $x$ in a local weighted polyhedral chart of $\Sigma$ at $x$. Then
    \[
    c_1(L|_X,\lVert\cdot\rVert)^{\wedge d}|_{\Delta\cap U(x)}=(\trop|_{\Delta\cap U(x)})^*(\operatorname{MA}_{\operatorname{poly},C_x}(g_{\trop})|_{\trop(\Delta\cap U(x))}).
    \]
\end{introthm}

On the maximal-dimensional faces, the polyhedral Monge-Amp{\`e}re measure reduces to the usual real Monge-Amp{\`e}re measure in the sense of Alexandrov. We write $\operatorname{MA}_\R(f)$ for the real Monge-Amp{\`e}re measure of a convex function $f$. We therefore obtain the following more classical form of the main theorem.

\begin{introcor}\label{cor:intro maximal}
    If in the statement of \Cref{thm:intro main}, the open face $\Delta$ has dimension $d=\operatorname{dim}(X)$, then the previous Monge-Amp{\`e}re measure can be written as
    \[
    c_1(L|_X,\lVert\cdot\rVert)^{\wedge d}|_{\Delta\cap U(x)}= d!\, m_\Delta\, \operatorname{MA}_\R(g_{\trop}|_{\trop(\Delta\cap U(x))}\circ\trop|_{\Delta\cap U(x)}),
    \]
    where $m_\Delta$ is the weight of $\Delta$.
\end{introcor}

When $X=Y$, \Cref{cor:intro maximal} recovers the description of the Monge-Amp{\`e}re measure of a semipositive toric metric on a toric variety in \cite[Theorem 4.8.11]{BPS}. Related comparisons on maximal-dimensional open faces of skeleton associated with polystable formal models were obtained by Vilsmeier \cite{Vilsmeier21}; \Cref{cor:intro maximal} gives an analogue for the tropical skeleton associated with the fixed tropicalization considered here.

We further show that, for a special class of semipositive toric metrics, the polyhedral Monge-Amp{\`e}re measure in the main theorem can be replaced by the usual real Monge-Amp{\`e}re measure even on lower-dimensional faces; see \Cref{cor:special class}.

Together with the Bedford-Taylor product comparison above, \Cref{thm:intro main} also gives a precise compatibility, through tropicalization, among three Monge-Amp{\`e}re constructions arising from different viewpoints: the Chambert-Loir construction from algebraic geometry, the polyhedral construction from tropical intersection theory, and the analytic Bedford-Taylor construction from pluripotential theory.

Most of these comparison results concern tropical pushforwards. Our main result goes further: it gives a local tropical pullback description of the original non-archimedean Monge-Amp{\`e}re measure, thus describing its local structure in terms of explicit tropical and convex data. This pullback description is compatible with the corresponding pushforward identities via the projection formula.

Unlike in the abelian case \cite{GubCanonical,GSMongeAmpere}, the tropical cycle in our formula depends on the local germ. The locality of this description arises from two phenomena of tropicalizations.

First, even at the level of a tropical variety, the contribution on a lower-dimensional open face is not determined by the restriction of the tropical Green function to that face alone. Tropical intersection theory also involves the surrounding higher-dimensional faces, their weights, and the transverse behaviour of the tropical Green function. This is illustrated by the counterexample in \Cref{counterexample: lower dimensional face}. Thus, a facewise description requires the weighted tropical variety around the point.

Second, the tropicalization of an arbitrary neighborhood is not an intrinsic local object. Since the tropicalization map on the tropical skeleton is finite-to-one and need not be injective, different local branches may have overlapping images, and the resulting weighted tropical variety may depend on the chosen neighborhood, as illustrated in \Cref{example: different branches}. After shrinking sufficiently around $x$, however, its weighted germ at $\trop(x)$ is independent of further shrinking. This intrinsic local object is the tropicalization of the analytic germ, which explains its role in our main theorem.

There is also a topological difference between the two settings: the support of the canonical measure in the abelian case is compact, while the tropical skeleton of the torus locus in our setting is non-compact.

\paragraph{Main ingredients.}

The main technical tool in our proof is the theory of $\delta$-forms. Building on the theory of real differential forms and currents on Berkovich spaces introduced by Chambert-Loir and Ducros \cite{2Antoine}, Gubler and K{\"u}nnemann introduced $\delta$-forms by combining non-archimedean differential forms with tropical intersection theory \cite{firstDeltaForm}. The resulting formalism simultaneously extends the calculus of differential forms and the intersection theory of tropical cycles: smooth forms occur as $\delta$-forms of codimension zero, while tropical cycles give rise to polyhedral currents, and the wedge product extends the tropical intersection product. In particular, first Chern currents associated with formal metrics can be represented by $\delta$-forms, and their top-degree wedge products recover the corresponding non-archimedean Monge-Amp{\`e}re measures. 

Mihatsch subsequently developed a more flexible tropical intersection formalism \cite{MihatschTIT} and extended the theory of $\delta$-forms from algebraic varieties to good Berkovich spaces using tropical spaces and their skeletons \cite{DeltaLubinTate}. The theory has since found applications to formal intersection theory and Lubin-Tate spaces \cite{DeltaLubinTate}, as well as to tropical formulas for non-archimedean local heights \cite{newArxivGGK}.

A key feature of Mihatsch's formalism is the explicit characterization of the pullback of a $\delta$-form on a tropical space in terms of a linear map with finite fibers; see \cite[(2.40)]{DeltaLubinTate}. In our setting, the required finite-to-one map is provided by the tropicalization itself. More precisely, Ducros \cite{Duc12} shows that the restriction of the map $\trop$ to the tropical skeleton is finite-to-one and, after choosing suitable polyhedral structures, restricts to a linear isomorphism on each face. We may therefore apply Mihatsch's description directly to the tropicalization map. It expresses the pushforward of the relevant $\delta$-form on the tropical skeleton as a tropical intersection product on the tropicalization, while the facewise linear isomorphisms allow us to recover the measure on each individual face. This interplay between the $\delta$-form formalism and the polyhedral structure of the tropical skeleton is the main mechanism behind the proof of our main theorem.

We first prove our results for semipositive toric model metrics, where this $\delta$-form description applies directly. The general case is then obtained by approximation by toric model metrics and the weak continuity of the Monge-Amp{\`e}re measures on both the non-archimedean and polyhedral sides.

\paragraph{Guideline to this paper.}

We begin by fixing notation in \Cref{section:notation} and recalling the necessary background on Berkovich analytic spaces in \Cref{section: berkovich preli}. Toric varieties, tropicalizations, and toric metrics are reviewed in \Cref{Section: Toric preliminaries}, followed in \Cref{Section: tropical skeletons} by the theory of tropical skeletons and the local polyhedral structures used later. In \Cref{section:delta forms preli}, we recall the theory of $\delta$-forms and establish two auxiliary lemmas needed in the proof of the main theorem. \Cref{Section: NA MA measure as delta current} recalls non-archimedean Monge-Amp{\`e}re measures and establishes an approximation by toric model metrics, describes the model case in terms of $\delta$-forms, and proves that the measure on the torus locus is supported on the tropical skeleton. In \Cref{Section: poly MA}, we recall the polyhedral Monge-Amp{\`e}re theory of \cite{MApoly}, develop the $\Gamma$-rational version needed here, and verify that it applies to the tropical Green functions arising in our setting. \Cref{Section: tropicalized MA} studies tropicalized Monge-Amp{\`e}re measures in the language of $\delta$-forms, establishes the real Monge-Amp{\`e}re description on maximal-dimensional open faces for arbitrary convex functions, and discusses both the obstruction on lower-dimensional faces and a special class of metrics for which a real Monge-Amp{\`e}re description remains available. Finally, \Cref{Section: main thm} contains the proof of our main theorem and its consequences.

\paragraph{Acknowledgement.}
I am very grateful to my PhD supervisor Walter Gubler for suggesting the topic of this paper, for many helpful weekly discussions throughout the project, and for his careful reading of earlier drafts. I would also like to thank my co-supervisor Roberto Gualdi for valuable discussions and guidance. I thank L{\'e}onard Pille-Schneider for a helpful conversation about polyhedral Monge-Amp{\`e}re measures.

\section{Notation and conventions}\label{section:notation}

\begin{enumerate}[label=(2.\arabic*)]
    \item Let $K$ be an algebraically closed complete field with a non-trivial non-archimedean absolute value $|\cdot |$ and the corresponding valuation $v=-\operatorname{log}|\cdot|$, let $\Gamma\coloneqq v(K^\times)$ be the value group and let $K^\circ$ be the valuation ring of $K$. 

    \item A \textit{variety} over $K$ is an integral separated scheme of finite type over $\operatorname{Spec}K$.

    \item An \textit{analytic space} over $K$ is a Hausdorff good $K$-analytic space in the sense of Berkovich.
    
    \item Let $N\cong\Z^n$ and $M\coloneqq N^\vee=\operatorname{Hom}(N,\Z)$. Let $N_\R$ (resp. $M_\R$) be the $\R$-vector space generated by $N$ (resp. $M$). 

    \item All fans considered in this paper are rational, i.e. their cones are rational polyhedral cones.

    \item A subset $C\subset N_\R$ is called a \textit{polyhedral set} if it is a locally finite union of polyhedra.
    
    \item\label{convention:normalization of MA} For every convex function $f:N_\R\to\R$, let $\operatorname{MA}_\R(f)$ be the \textit{real Monge-Amp{\`e}re measure} defined by
    \[
    \operatorname{MA}_\R(f)(E)\coloneqq\operatorname{Vol}(\partial f(E))
    \]
    for a Borel set $E$ of $N_\R$, where $\partial f(E)$ is the \textit{sub-differential} of $f$ at $E$: for every point $u\in N_\R$, we set
    \[
    \partial f(u)\coloneqq\left\{ x\in M_\R:\langle x,v-u \rangle\leq f(v)-f(u)\text{ for all }v\in N_\R \right\}
    \]
    and
    \[
    \partial f(E)\coloneqq \bigcup_{u\in E}\partial f(u).
    \]
    Let $\sigma$ be a face of a tropical cycle in $N_\R$. Let
    \[
    N_{\sigma,\R}=\operatorname{span}_\R\{x-y:x,y\in\sigma\}\quad\text{and}\quad N_\sigma=N\cap N_{\sigma,\R}.
    \]
    We denote by $\operatorname{MA}_\R(f|_\sigma)$ the real Monge-Amp{\`e}re measure computed with the integral structure induced by $N_\sigma$.

    \item Let $f:N_\R\to\R$ be a \textit{piecewise affine function} on $N_\R$. Then there is a complete polyhedral complex $\Pi$ in $N_\R$ such that for every $\sigma\in\Pi$ we have
    \[
    f|_\sigma(x)=\langle m_\sigma,x \rangle +\ell_\sigma
    \]
    with some $(m_\sigma,\ell_\sigma)\in M_\R\times\R$. We say $f$ has \textit{rational slopes} if $m_\sigma\in M_\Q$ for every $\sigma\in\Pi$. We say $f$ is $\Gamma$-\textit{lattice} if $(m_\sigma,\ell_\sigma)\in M\times\Gamma$ for every $\sigma\in\Pi$. If $af$ is $\Gamma$-lattice for some positive integer $a$, then we call $f$ a $\Gamma$-\textit{rational} function.

    \item We say a polyhedron $\Delta$ in $N_\R$ is $\Gamma$-\textit{rational} if there are $m_i\in M$ and $\ell_i\in\Gamma$ such that
    \[
    \Delta=\bigcap_{i=1}^r\{u\in N_\R:\langle m_i,u\rangle\geq\ell_i\}.
    \]
    If a polyhedral complex consists only of $\Gamma$-rational polyhedra, we say it is a $(\Z,\Gamma)$-polyhedral complex. The notions of a $(\Z,\Q)$-polyhedral complex and a $(\Z,\R)$-polyhedral complex are defined similarly, allowing $\ell_i\in\Q$ and $\ell_i\in\R$, respectively.

    \item Let $C\subset N_\R$ be a pure $d$-dimensional weighted $(\Z,\R)$-polyhedral complex, where every $d$-dimensional face $\sigma$ of $C$ has weight $m_\sigma$. For a codimension-one face $\tau$ of $C$ and a $d$-dimensional face $\sigma\supset\tau$, let $u_{\sigma/\tau}$ denote the primitive generator of $N_\sigma/N_\tau$ pointing from $\tau$ into $\sigma$. We say that $C$ is \textit{balanced} if
    \[
    \sum_{\sigma\supset\tau} m_\sigma u_{\sigma/\tau}=0\quad\text{in }N/N_\tau
    \]
    for every codimension-one face $\tau$.

    \item A $d$-dimensional \textit{tropical cycle} in $N_\R$ is an equivalence class of balanced weighted $(\Z,\R)$-polyhedral complexes, where two such complexes are equivalent if they admit a common weighted refinement. If it can be represented by a balanced weighted $(\Z,\R)$-polyhedral complex with positive weights, we call it a \textit{tropical variety}.

    \item For another lattice $N'$ with associated $\R$-vector space $N'_\R$, we say a function $f:N_\R'\to N_\R$ is $\Gamma$-lattice if the pullback of every $\Gamma$-lattice affine function on $N_\R$ is a $\Gamma$-lattice affine function on $N_\R'$. If furthermore $f$ is bijective and its inverse is again a $\Gamma$-lattice affine map, we call $f$ a $\Gamma$-lattice isomorphism.

    \item Recall that every polyhedron $\sigma$ of an abstract polyhedral complex is in a real vector space $\R^{n_\sigma}$. An abstract polyhedral complex is called an abstract $(\Z,\Gamma)$-polyhedral complex if furthermore the induced affine map $\R^{n_\tau}\to\R^{n_\sigma}$ is $\Gamma$-lattice.

    A $\Gamma$-integral affine structure is given by the polyhedra of a $(\Z,\Gamma)$-abstract polyhedral complex, well-defined up to $(\Z,\Gamma)$-subdivisions.
\end{enumerate}

\section{Analytic spaces}\label{section: berkovich preli}

By \textit{analytic space over} $K$, or simply an \textit{analytic space}, we mean a Hausdorff \textit{good} $K$-analytic space as defined in \cite{Ber93}. These are analytic spaces in the sense of Berkovich \cite{Ber90}. An analytic space $\mathcal{X}$ is \textit{good} if every $x\in\mathcal{X}$ has a neighborhood isomorphic to the Berkovich spectrum $\mathcal{M}(\mathcal{A})$ for an affinoid algebra $\mathcal{A}$. A good analytic space is called \textit{strictly analytic} if these affinoid algebras can be chosen strictly affinoid. Analytic domains of this form are called \textit{strictly affinoid domains}. Throughout this paper, strictly analytic spaces are endowed with the $G$-topology generated by strictly affinoid domains \cite[Section 1.3]{Ber93}.

Every algebraic variety $X$ over $K$ has a Berkovich analytification $X^{\an}$. We first define it in the affine case. If $X=\operatorname{Spec}(A)$, with $A$ a finitely generated $K$-algebra, then $X^{\an}$ is the space of multiplicative seminorms of $A$ extending the absolute value $|\cdot|$ of $K$, endowed with the topology generated by the functions $x\mapsto|a(\cdot)|$ for all $a\in A$. One can define a sheaf of analytic functions on $X^{\an}$ as in \cite[Section 1.5 and Remark 3.4.2]{Ber90} to get a locally ringed affine analytic space. In general, the analytification $X^{\an}$ is defined by gluing the affine analytic spaces. It follows from \cite[Section 3.4]{Ber90} that $X^{\an}$ is a strictly analytic space. As we mentioned above, we will use the $G$-topology on $X^{\an}$ generated by strictly affinoid domains.

Let $L$ be a line bundle on $X$. A metric $\lVert\cdot\rVert$ on $L^{\an}$ is called a \textit{model metric} if it is induced by a proper model $(\mathfrak{X},\mathfrak{L},e)$ of $(X,L)$, see for example \cite[Definition 1.3.5]{BPS}. We say a model metric is \textit{nef} if $\mathfrak{L}$ can be chosen as a nef line bundle. A continuous metric on $L^{\an}$ is called \textit{semipositive} if it is a uniform limit of nef model metrics on $L$.

\section{Toric varieties and tropicalizations}\label{Section: Toric preliminaries}

\paragraph{Basic toric notation.} We briefly recall the toric notation that will be used later. For a general reference on toric varieties, we refer to \cite{Ful93}.

Let $T=\operatorname{Spec}K[M]$ be an algebraic torus. A \textit{toric variety} $Y$ over $K$ is a normal variety over $K$ which contains $T$ as a dense open subset, such that the multiplication on $T$ extends to an action of $T$ on $Y$.

The geometry of toric varieties is encoded by convex geometry. Every strongly convex rational polyhedral cone $\sigma\subset N_\R$ gives us an affine toric variety $Y_\sigma=\operatorname{Spec}K[M\cap\sigma^\vee]$, and every affine toric variety with torus $T$ arises in this way. Gluing those affine toric varieties along the common faces, we have a bijection $\Sigma\mapsto Y_\Sigma$ between the set of fans in $N_\R$ and the set of isomorphism classes of toric varieties with torus $T$. If $\Sigma$ is complete, i.e., $|\Sigma|=\cup_{\sigma\in\Sigma}\sigma=N_\R$, then $Y_\Sigma$ is proper.

Let $\Sigma$ be a fan in $N_\R$ and let $Y=Y_\Sigma$ be its corresponding toric variety. The correspondence $\psi\mapsto D_\psi$ gives a bijection between integral piecewise linear support functions on $\Sigma$ and $T$-invariant Cartier divisors on $Y_\Sigma$. Equivalently, via $D\mapsto\left(O(D),s_D\right)$, such support functions are in bijection with toric line bundles equipped with a toric section.

\paragraph{Tropicalizations.} Let $\Sigma$ be a fan in $N_\R$ and let $Y=Y_\Sigma$ be the associated toric variety. For $\sigma\in\Sigma$, we set $N(\sigma)\coloneqq N_\R/\langle\sigma\rangle_\R$, where $\langle\sigma\rangle_R$ is the $\R$-linear space generated by $\sigma$. We also put $N_\sigma\coloneqq\coprod_{\tau\prec\sigma}N(\tau)$. The topology on $N_\sigma$ is described in \cite[Section I.1]{topologyOfNsigma}, which also contains a helpful illustration. Globally, we consider $N_\Sigma=\coprod_{\sigma\in\Sigma}N(\sigma)$ whose topology is obtained by gluing the spaces $N_\sigma$. Note that $N_\R$ is a dense subset in $N_\Sigma$.

By \cite{Pay}, we have a tropicalization map
\[
\trop:Y_\Sigma^{\an}\to N_\Sigma
\]
which is locally given by
\begin{align*}
    \trop_\sigma:Y_\sigma^{\an}&\to N_\sigma=\operatorname{Hom}(S_\sigma,\R\cup\{\infty\})\\
    p&\mapsto (u\mapsto -\operatorname{log}|p(\chi^u)|).
\end{align*}
As shown in \cite[Section 4.1]{BPS}, there exists a canonical section $\iota:N_\Sigma\to Y_\Sigma^{\an}$ of $\trop$ that is a homeomorphism onto a closed subset of $Y_\Sigma$. Therefore, we call $N_\Sigma$ the canonical skeleton of $Y_\Sigma$.

The following fact allows us to obtain strictly affinoid domains from polyhedral subsets of the tropicalization. Let $\Delta$ be a $\Gamma$-rational polytope in $N_\R$. Then the preimage $\trop^{-1}(\Delta)$ under the tropicalization map is a strictly affinoid domain by \cite[Lemma 6.21]{guide}, which is a corollary of \cite[Proposition 6.9]{JRabTropAnalytic}.

\paragraph{Toric metrics.}
Let $L=\mathcal{O}(D)$ be a toric line bundle on $Y$, with canonical toric section $s_D$. We now recall the notion of \textit{toric metrics} on $L^{\an}$, following \cite[Section 4]{BPS} in the case where $v$ is discrete and \cite[Section 2.5]{ToricMetricNonDiscrete} in the general case.

We denote the support function corresponding to $D$ by $\psi$. Let $\lVert\cdot\rVert$ be a semipositive toric metric on $L^{\an}$. Then there exists a continuous \textit{tropical Green function} $g_{\trop}:N_\R\to\R$ such that $g_{\trop}\circ\trop=-\operatorname{log}\lVert s_D\rVert$. By \cite[Theorem 4.8.1]{BPS} and \cite[Theorem 2.5.8]{ToricMetricNonDiscrete}, the function $g_{\trop}$ is a convex function such that $|g_{\trop}+\psi|$ is bounded.  If furthermore $\lVert\cdot\rVert$ is a semipositive toric model metric, then it follows from \cite[Theorem 4.5.10]{BPS} and \cite[Proposition 2.5.5]{ToricMetricNonDiscrete} that $g_{\trop}$ is a $\Gamma$-rational piecewise affine function. 
Moreover, by \cite[Remark 4.5.8 and Corollary 4.5.9]{BPS} and \cite[Remark 2.5.9]{ToricMetricNonDiscrete}, this gives us a bijection between semipositive toric model metrics $\lVert\cdot\rVert$ and convex $\Gamma$-rational piecewise affine tropical Green functions $g_{\trop}$ such that $|g_{\trop}+\psi|$ is bounded.

\section{Tropical skeletons}\label{Section: tropical skeletons}

In this section, we denote by $\mathbb{G}_m^r$ the $r$-dimensional analytic torus. Let $\mathcal{X}$ be an analytic space over $K$ of dimension $d$ and $f:\mathcal{X}\to\mathbb{G}_m^n$ be a morphism. Then we have a tropicalization map
\[
\trop_f=\trop\circ f:\mathcal{X}\to\R^n.
\]
We introduce Ducros' tropical skeletons \cite{Duc12} in this section, as our Monge-Amp{\`e}re measures will be supported on the tropical skeletons. 

Write $f=(f_1,\ldots,f_n)$. First we consider the closed subspace
\[
\Sigma'(\mathcal{X},f)\coloneqq\bigcup_{1\leq i_1,\ldots,i_d\leq n}(f_{i_1},\ldots,f_{i_d})^{-1}\Sigma(\mathbb{G}_m^d),
\]
where $\Sigma(\mathbb{G}_m^d)$ is the standard skeleton consisting of weighted Gauss norms. There is also an alternative way to describe $\Sigma'(\mathcal{X},f)$ as the collection of some Abhyankar points on $\mathcal{X}$. We refer to \cite[(0.13)]{Duc12} for this equivalent definition.

It follows from \cite[Theorem 5.1]{Duc12} that $\Sigma'(\mathcal{X},f)$ has a canonical structure of a piecewise linear space of dimension $\leq d$, i.e. it is Hausdorff and locally isomorphic to some polyhedral subsets $C\subset\R^s$ together with piecewise affine functions on $C$. See \cite[Section 2.1]{DeltaLubinTate} for the precise definition of piecewise linear spaces and piecewise linear maps. In particular, if $\trop_f$ is proper, then the restriction of $\trop$ to $\Sigma'(\mathcal{X},f)$ defines a piecewise linear map
\[
\trop_f|_{\Sigma'(\mathcal{X},f)}:\Sigma'(\mathcal{X},f)\to\trop_f(\mathcal{X}).
\]

We define the tropical skeleton $\Sigma(\mathcal{X},f)$ to be the $d$-dimensional locus of $\Sigma'(\mathcal{X},f)$. One can endow $\Sigma(\mathcal{X},f)\backslash\partial\mathcal{X}$ with the structure of a \textit{tropical space} by \cite[Corollary 3.11(2)]{DeltaLubinTate}, where a tropical space is a weighted piecewise linear space with \textit{linear functions} satisfying some balancing conditions \cite[Definition 2.34 and Corollary 2.35]{DeltaLubinTate}. Here the term ``linear function'' refers to Mihatsch's sheaf of linear functions \cite[Definition 2.25]{DeltaLubinTate}, rather than to the full sheaf of piecewise linear functions on the underlying piecewise linear space. Thus, although $\trop_f|_{\Sigma'(\mathcal{X},f)}$ is a piecewise linear map on the level of piecewise linear spaces, it is also a linear map on the tropical space $\Sigma(\mathcal{X},f)\backslash\partial\mathcal{X}$ \cite[Definition 3.8(3)]{DeltaLubinTate}. 

If $\trop_f(\mathcal{X})$ is of pure dimension $d$, then
\begin{equation}\label{eq:trop on skeleton which is surj}
    \trop_f|_{\Sigma(\mathcal{X},f)}:\Sigma(\mathcal{X},f)\to\trop_f(\mathcal{X})
\end{equation}
is surjective by \cite[(3.8)]{DeltaLubinTate}.

\paragraph{Tropical skeletons on compact analytic spaces.}
Recall that a $\Gamma$-\textit{integral affine structure} on a space $\mathcal{X}$ is an equivalence class, up to $(\Z,\Gamma)$-subdivision, of homeomorphisms $|\Pi|\cong\mathcal{X}$, where $\Pi$ is a locally finite abstract $(\Z,\Gamma)$-polyhedral complex.

If $\mathcal{X}$ is a compact strictly analytic space of pure dimension $d$, then as stated in \cite[Remark 3.20]{TropicalLinearIso}, we have a unique way to endow $\Sigma(\mathcal{X},f)$ with a $\Gamma$-integral affine structure such that $\trop_f|_{\Sigma'(\mathcal{X},f)}$ restricts to a $\Gamma$-lattice isomorphism on every face: following \cite[8.15]{TropicalLinearIso}, after subdivisions, we can take an abstract $(\Z,\Gamma)$-polyhedral complex with support $\Sigma(\mathcal{X},f)$ and a $(\Z,\Gamma)$-polyhedral complex $\Pi$ with support $\trop_f(\mathcal{X})$ such that the following holds. 
\begin{itemize}
    \item For any $d$-dimensional face $\Delta$ of $\Sigma(\mathcal{X},f)$, let $\sigma\coloneqq\trop_f(\Delta)$. Then $\sigma$ is a face of $\Pi$ and $\trop_f|_\Delta:\Delta\xrightarrow{\sim}\sigma$ is a $\Gamma$-lattice isomorphism.
    \item The tropical skeleton $\Sigma(\mathcal{X},f)$ admits the structure of a weighted abstract polyhedral complex: for every $d$-dimensional face $\Delta$ of $\Sigma(\mathcal{X},f)$, we can take $x\in\operatorname{relint}(\Delta)$ and set the weight of $\Delta$ to be $m_\Delta\coloneqq m_{\trop}(\mathcal{X},x)$ the tropical weight at $x$ given in \cite[Remark 4.13]{TropicalLinearIso}, which is well-defined as explained in \cite[Section 4]{TropicalLinearIso}.
    \item We can endow $\Pi$ with the structure of a weighted polyhedral complex, where the weight of a maximal-dimensional face $\sigma$ is denoted by $m_\sigma$. More precisely, the weights are given by push-forward, i.e.
    \[
    (\trop_f)_*(\Sigma(\mathcal{X},f))=\trop_f(\mathcal{X})
    \]
    and
    \begin{equation}\label{eq:local weights pushforward}
        m_\sigma=\sum_{\substack{\Delta\text{ a }d\text{-face of }\Sigma(\mathcal{X},f)\\ \trop_f(\Delta)=\sigma}}m_\Delta.
    \end{equation}
\end{itemize}  

\paragraph{Tropical skeletons of closed subvarieties of tori.}\label{subsection:local charts of tropical skeletons}

In our case, we are mainly interested in the case that $\mathcal{X}$ is the Berkovich analytification of a closed subvariety of the torus $T$. Let $X$ be an algebraic variety of dimension $d$ over $K$ with a closed embedding $\phi:X\hookrightarrow T$. Let $X^{\an}$ and $T^{\an}$ be the Berkovich analytifications of $X$ and $T$. Let $\phi^{\an}:X^{\an}\hookrightarrow T^{\an}$ be the corresponding morphism between analytic spaces. It follows from \cite[Theorem 3.4.1]{Ber90} that $\partial X^{\an}=\varnothing$. By \cite[Proposition 3.7]{DeltaLubinTate}, we have that $\Sigma'(X^{\an},\phi^{\an})$ is purely of dimension $d$ and thus $\Sigma'(X^{\an},\phi^{\an})=\Sigma(X^{\an},\phi^{\an})$. In this case, the map $\trop_{\phi^{\an}}$ is proper and hence we have a piecewise linear map
\[
\trop_{\phi^{\an}}|_{\Sigma(X^{\an},\phi^{\an})}:\Sigma(X^{\an},\phi^{\an})\to\trop_{\phi^{\an}}(X^{\an}).
\]

Note that in our case $X$ is not necessarily proper and thus $X^{\an}$ is not necessarily compact. We shall work locally on the tropical skeleton $\Sigma(X^{\an},\phi^{\an})$.

A \textit{local} $(\Z,\Gamma)$\textit{-polyhedral chart of} $\Sigma(X^{\an},\phi^{\an})$ \textit{compatible with} $\trop_{\phi^{\an}}$ is a weighted abstract $(\Z,\Gamma)$-polyhedral complex $\mathcal{P}$ such that
\begin{itemize}
    \item $|\mathcal{P}|$ is a compact subset of $|\Sigma(X^{\an},\phi^{\an})|$ with pure dimension $d$, 
    \item there exists a $(\Z,\Gamma)$-polyhedral complex $\Pi$ with support $\trop_{\phi^{\an}}(\mathcal{P})$ such that, for any $d$-face $\Delta$ of $\mathcal{P}$ and $\sigma\coloneqq\trop_{\phi^{\an}}(\Delta)$, we have $\sigma$ is a face of $\Pi$ and $\trop_{\phi^{\an}}:\Delta\xrightarrow{\sim}\sigma$ is a $\Gamma$-lattice linear isomorphism, and
    \item every $d$-face of $\mathcal{P}$ is endowed with the weight introduced by the weighted tropical skeleton $\Sigma(X^{\an},\phi^{\an})$.
\end{itemize}
If $\mathcal{P}$ is a compact neighborhood of $x$, then we say that $\mathcal{P}$ is a \textit{local} $(\Z,\Gamma)$\textit{-polyhedral chart of} $\Sigma(X^{\an},\phi^{\an})$ \textit{at} $x$ \textit{compatible with} $\trop_{\phi^{\an}}$.

For every $x\in\Sigma(X^{\an},\phi^{\an})$, there exists a local $(\Z,\Gamma)$-polyhedral chart of $\Sigma(X^{\an},\phi^{\an})$ at $x$ compatible with $\trop_{\phi^{\an}}$. 
Indeed, let $\{\Delta_i\}_{i\in\N_{>0}}$ be a countable covering of $N_\R$ by full-dimensional $\Gamma$-rational polytopes such that $\Delta_i\subset\Delta_{i+1}$ for all $i$. Then $V_i\coloneqq\trop_{\phi^{\an}}^{-1}(\Delta_i)$ is a compact strictly affinoid space over $K$. By the above discussion of \cite[8.15]{TropicalLinearIso}, we have that $\{\Sigma(V_i,\phi^{\an}|_{V_i})\}_{i\in\N_{>0}}$ is a covering of $\Sigma(X^{\an},\phi^{\an})$ by local $(\Z,\Gamma)$-polyhedral charts compatible with $\trop_{\phi^{\an}}$. Therefore, there exists an $i\in\N_{>0}$ such that $\Sigma(V_i,\phi^{\an}|_{V_i})$ is a local $(\Z,\Gamma)$-polyhedral chart of $\Sigma(X^{\an},\phi^{\an})$ at $x$ compatible with $\trop_{\phi^{\an}}$.

\begin{rmk}
    One may also ask whether the whole tropical skeleton $\Sigma(X^{\an},\phi^{\an})$ admits the structure of a locally finite abstract $(\Z,\Gamma)$-polyhedral complex. It is natural to expect such a structure by modifying $\Sigma(V_i,\phi^{\an}|_{V_i})$ above and then passing to compatible subdivisions. We do not pursue this global statement here, since our description of Monge-Amp{\`e}re measures will be local and the local charts considered above are sufficient for our purpose.
\end{rmk}

\section{Delta-forms and delta-currents}\label{section:delta forms preli}

This section collects some preliminaries on $\delta$-forms and $\delta$-currents, since we will study the Monge-Amp{\`e}re measures as $\delta$-currents. 

\paragraph{Delta-forms.}
Gubler and K{\"u}nnemann introduced $\delta$-forms on the non-archimedean analytification of an algebraic variety, by extending the space of real-valued differential forms on the Berkovich space \cite{firstDeltaForm}. Mihatsch subsequently developed a theory of tropical intersection in which $\delta$-forms on Euclidean spaces play a central role \cite{MihatschTIT}, and extended this formalism to analytic spaces without boundary \cite{DeltaLubinTate}. 

The resulting theory of $\delta$-forms simultaneously generalizes tropical cycles and the theory of real differential forms developed by Lagerberg \cite{Lagerberg}. From the tropical side, it admits a $\wedge$-product extending the tropical intersection product. From the differential-form side, it admits the differentials $\mathrm d'$ and $\mathrm d''$.

By Mihatsch's definition \cite[Definition 4.2]{DeltaLubinTate}, every $\delta$-form is supported on a \textit{skeleton}, which is locally a closed piecewise linear space of some tropical skeletons. As introduced in \Cref{Section: tropical skeletons}, in a boundary-less analytic space, every tropical skeleton is a tropical space. Hence, let us recall the definition of $\delta$-forms on tropical spaces following \cite[Section 2.7]{DeltaLubinTate}.

First of all, a $\delta$-form on $\R^r$ is a polyhedral current satisfying balancing conditions \cite[Section 3]{MihatschTIT}.

Let $X$ be a tropical space. Then a $\delta$-form on $X$ is locally of the form $f^\star\gamma$, where $f$ is a linear function to $\R^r$ and $\gamma$ is a $\delta$-form on $\R^r$. An explicit description of $f^\star\gamma$ is given in \cite[Proposition 2.38]{DeltaLubinTate}: if $f:X\to\R^r$, then $f^\star\gamma$ is the unique polyhedral current on $X$ such that for all compact polyhedral subset $C$ of $X$ and all functions $g:C\to\R^s$ with the properties that
\begin{itemize}[itemsep=0pt]
    \item $g$ has finite fibers, 
    \item there exists an affine linear map $p:\R^s\to\R^r$ such that $f=p\circ g$,
\end{itemize}
we have the following identity:
\begin{equation}\label{eq: pullback description}
    g_*\left( (f^\star\gamma)\big|_{C\backslash g^{-1}(g(\partial C))} \right)=g_*(C)\wedge p^*\gamma,
\end{equation}
away from $g(\partial C)$, where $\partial C$ denotes the boundary of $C$ in $X$. 

We emphasize that the weights of the tropical space $X$ are already encoded in the definition of $f^\star\gamma$. 

Suppose $X$ is an analytic space with $\partial X=\varnothing$ of pure dimension. Then a $\delta$-form on $X$ is specified by local data consisting of $\delta$-forms on tropical skeletons that are compatible on overlaps. As in \cite[Definition 4.2]{DeltaLubinTate}, we write $\trop_f^\star(\gamma)$ for the $\delta$-form $\Sigma\mapsto\trop_f^\star(\gamma)|_\Sigma$, where $\Sigma$'s are skeletons contained in an open subset $U$ of $X$ and the $\delta$-form is given by $f:U\to\mathbb{G}_m^r$ and a $\delta$-form $\gamma$ on $\R^r$.

We discuss the restriction of $\delta$-forms to polyhedral subsets.
The next lemma shows that $f^\star$ commutes with restriction to open polyhedral subsets.

\begin{lem}[Restriction Lemma]\label{lem:restriction}
    Let $X$ be a tropical space. Let $f:X\to\R^r$ be a linear map and $\gamma$ be a $\delta$-form on $\R^r$. Then 
    \[
    (f^\star\gamma)|_\Delta=(f|_\Delta)^\star(\gamma)
    \]
    for every open polyhedral subset $\Delta$ of $X$.
\end{lem}
\begin{proof}
    Since $\Delta$ is an open polyhedral subset of $X$, then $\Delta$ is also a tropical space, by \cite[Example 2.26(2) and Corollary 2.35]{DeltaLubinTate}. Therefore, we have that $(f|_\Delta)^\star$ is well-defined.

    We are going to use (\ref{eq: pullback description}) to get an explicit description of $(f^\star\gamma)
    |_\Delta$. Suppose now that $H$ is a compact polyhedral subset of $\Delta$. Let $g:H\to\R^s$ be a map with finite fibers such that $f|_H=q\circ g$ for an affine function $q:\R^s\to\R^r$. As $H$ is also a compact polyhedral subset of $X$, it follows from (\ref{eq: pullback description}) that
    \[
    g_*\left( (f^\star\gamma)\big|_{H\backslash g^{-1}(g(\partial H))} \right)=g_*(H)\wedge q^*\gamma
    \]
    away from $g(\partial_X H)$, where $\partial_X H$ is the boundary of $H$ in $X$. Since $H\subset\Delta$, we have
    \[
    \big( (f^\star\gamma)|_\Delta \big)\big|_{H\backslash g^{-1}(g(\partial H))}=(f^\star\gamma)\big|_{H\backslash g^{-1}(g(\partial H))}.
    \]
    Thus,
    \[
    g_*\left( \big( (f^\star\gamma)|_\Delta \big)\big|_{H\backslash g^{-1}(g(\partial H))} \right)=g_*(H)\wedge q^*(\gamma).
    \]
    The right-hand side of this formula is exactly the explicit description (\ref{eq: pullback description}) of $(f|_\Delta)^\star(\gamma)$, as $f|_H=(f|_\Delta)|_H$. 
    
    Let $\partial_\Delta H$ denote the boundary of $H$ in $\Delta$. It remains to show that $\partial_X H=\partial_\Delta H$. Indeed, since $\Delta$ is open in $X$, we have
    \[
    \partial_\Delta H=\Delta\cap\partial_X H.
    \]
    Since $H$ is compact in the Hausdorff space $X$, we have that $H$ is closed in $X$ and thus
    \[
    \partial_X H\subset \overline{H}^X=H\subset\Delta,
    \]
    which implies that $\partial_X H=\partial_\Delta H$.
\end{proof}

We are particularly interested in the case where $\gamma=\ddc\phi$ for a piecewise affine function $\phi$ on $N_\R$, as such forms arise in the definition of the first Chern $\delta$-forms used in the next sections.

\begin{lem}\label{lem:f star ddc}
    Let $X$ be an analytic space with $\partial X=\varnothing$, purely of some dimension $n$. Let $f:X\to\mathbb{G}_m^r$ be a morphism and $\phi$ be a piecewise affine function on $\R^r$. Then
    \[
    \trop_f^\star(\ddc\phi)=\ddc(\phi\circ\trop_f).
    \]
\end{lem}
\begin{proof}
    Let $\Sigma\coloneqq\Sigma(U,g)$ be a tropical skeleton in $X$ defined by an open subset $U$ of $X$ and a morphism $g:U\to\mathbb{G}_m^s$. Then $\partial U=\varnothing$ and thus $\Sigma$ is a tropical space.

    It was stated in \cite[Example 2.45(1)]{DeltaLubinTate} that for piecewise smooth forms, the pullback definition for $\delta$-forms agrees with the usual pullback, which implies that
    \[
    \trop_f^\star(\phi)|_\Sigma=\trop_f^*(\phi)|_\Sigma.
    \]
    According to \cite[Proposition 2.39]{DeltaLubinTate}, the pullback of a $\delta$-form along a linear map commutes with the derivatives $\mathrm{d}'$ and $\mathrm{d}''$. Therefore,
    \[
    \trop_f^\star(\ddc\phi)|_\Sigma=\ddc(\trop_f^\star(\phi))|_\Sigma.
    \]
    Hence,
    \[
    \trop_f^\star(\ddc\phi)=\ddc(\trop_f^\star(\phi))=\ddc(\phi\circ\trop_f).
    \]
\end{proof}

\paragraph{Delta-currents.}
Following \cite[Definition 3.4(1)]{MihatschTIT}, we use $(p,q,r)$ to denote the tridegree of a $\delta$-form on $\R^n$ with codimension $r$ in the sense of polyhedral currents, and $(p,q)$ to denote the bidegree in the sense of currents. The trigrading and bigrading of $\delta$-forms on tropical spaces and boundary-less analytic spaces are given by the trigrading and bigrading of the corresponding $\delta$-forms on Euclidean spaces \cite[Definition 2.41 and Definition 4.2]{DeltaLubinTate}. 
For an open subset $U$ of $\R^n$ or of $X^{\an}$, where $X^{\an}$ is the analytification of a variety of dimension $n$, we write $B^{p,q}(U)$ for the space of $\delta$-forms of bidegree $(p,q)$ on $U$, and $B_c^{p,q}(U)$ for the subspace consisting of those with compact support in $U$.

A $\delta$-current on an open subset $U$ is a linear functional on the space of compactly supported $\delta$-forms on $U$, i.e.
\[
E^{p,q}(U)\coloneqq\operatorname{Hom}(B_c^{n-p,n-q}(U),\R)
\]
is the set of $(p,q)$-$\delta$-currents on $U$. For example, 
\begin{itemize}
    \item for every algebraic cycle $Z$ of codimension $p$ of $X$, we have the current of integration $\delta_Z$ on $X^{\an}$ defined by
    \[
    \delta_Z(\alpha)\coloneqq\int_{Z^{\an}}\alpha
    \]
    for $\alpha\in B_c^{n-p,n-p}(X^{\an})$;
    \item for every abstract polyhedral complex $C$, such as a tropical cycle, the current of integration along $C$ is, in the notation of \cite{MihatschTIT}, defined by
    \[
    \delta_C=\sum_\sigma m_\sigma\wedge[\sigma,\mu_\sigma],
    \]
    where $\sigma$ ranges over all maximal faces of $C$, $m_\sigma$ is the weight of $\sigma$ and $\mu_\sigma$ is the Lebesgue measure induced by the integral structure $N_\sigma$.
\end{itemize}
There is a product
\begin{align*}
    B^{p,q}(U)\times E^{p',q'}(U)&\to E^{p+p',q+q'}(U)\\
    (\alpha,T)&\mapsto \big( \alpha\wedge T:\beta\mapsto(-1)^{(p+q)(p'+q')}T(\alpha\wedge\beta) \big).
\end{align*}

\section{Non-Archimedean Monge-Ampère measures}\label{Section: NA MA measure as delta current}

Let $Y$ be a proper variety of dimension $n$ over the algebraically closed non-archimedean field $K$. Let $L$ be a line bundle over $Y$ and $\lVert\cdot\rVert$ be a metric on $L$. The Monge-Amp{\`e}re measure $c_1(L,\lVert\cdot\rVert)^{\wedge n}$ on $Y^{\an}$ was first introduced in \cite[Section 2]{ChambertLoir} and extended to the present setting in \cite[Section 3]{GubMAdef}. We briefly recall the construction of $c_1(L,\lVert\cdot\rVert)^{\wedge n}$. If $\lVert\cdot\rVert$ is a model metric, then the associated Monge-Amp{\`e}re measure can be studied using the theory of $\delta$-forms and $\delta$-currents, as will be explained later in this section. Suppose first that $\lVert\cdot\rVert$ is a model metric induced by the model $(\mathfrak{Y},\mathfrak{L},e)$ of $(Y,L)$. The proof of \cite[Proposition 3.5]{GaoZiyangHabeggerCLMA} allows us to assume that the special fiber $\mathfrak{Y}_s$ is reduced. Then for every irreducible component $V$ of $\mathfrak{Y}_s$, there is a unique point $\xi_V$ whose reduction is the generic point of $V$. Then
\[
c_1(L,\lVert\cdot\rVert)^{\wedge n}\coloneqq \frac{1}{e^n}\sum_V\operatorname{deg}_{\mathcal{L}}(V)\cdot\xi_V
\]
where $V$ ranges over all irreducible components of $\mathfrak{Y}_s$. For a continuous semipositive metric $\lVert\cdot\rVert$ on $L$, the associated Monge-Amp{\`e}re measure is the weak limit of the measures corresponding to any sequence of semipositive model metrics converging uniformly to $\lVert\cdot\rVert$; see \cite[Proposition 2.7(b)]{ChambertLoir} and \cite[Proposition 3.12]{GubMAdef}.

Suppose now $Y=Y_\Sigma$ is a proper toric variety with torus $T$ of dimension $n$ corresponding to a complete fan $\Sigma$ on $N_\R$. Then we have a tropicalization map $\trop:Y^{\an}\to N_\Sigma$ as introduced in \Cref{Section: Toric preliminaries}.

Let $X$ be a closed subvariety of $Y_\Sigma$ of dimension $d$ with a closed embedding $i:X\hookrightarrow Y$. 

Without loss of generality, we may assume that $X\cap T\neq\varnothing$. Indeed, if $X\cap T=\varnothing$, then we can take an orbit $O(\sigma)$ that contains the generic point of $X$. The closure of this orbit, $V(\sigma)=\overline{O(\sigma)}$, is then a proper toric variety with $O(\sigma)$ a dense open subset. We can then reduce to the case where $X\cap T\neq\varnothing$ by replacing $Y$ with $V(\sigma)$ and $T$ with $O(\sigma)$.

Let $L=\mathcal{O}(D)$ be a toric line bundle on $Y$ and let $s_D$ be the canonical toric section of $L$. Let $\lVert\cdot\rVert$ be a semipositive toric metric on $L$ and let $g_{\trop}:N_\R\to\R$ be the corresponding convex tropical Green function.

By the following lemma, we can approximate a Monge-Amp{\`e}re measure associated with a semipositive toric metric by Monge-Amp{\`e}re measures for semipositive toric model metrics.

\begin{lem}\label{lem:approximated semipos metric by model metrics}
    Let $\lVert\cdot\rVert$ be a semipositive toric metric on $L$ with convex tropical Green function $g_{\trop}:N_\R\to\R$. Then there exists a sequence of semipositive toric model metrics $(\lVert\cdot\rVert_i)_{i\geq 1}$ with tropical Green functions $(g_i)_{i\geq 1}$ such that
    \begin{itemize}
        \item $(g_i)_{i\geq 1}$ converges to $g_{\trop}$ uniformly on $N_\R$, and
        \item $\lVert\cdot\rVert$ is the uniform limit of $(\lVert\cdot\rVert_i)_{i\geq 1}.$
    \end{itemize}
    In particular, the measure $c_1(L|_X,\lVert\cdot\rVert)^{\wedge d}$ is the limit of $(c_1(L|_X,\lVert\cdot\rVert_i)^{\wedge d})_{i\geq 1}$ by weak convergence.
\end{lem}
\begin{proof}
    We denote by $\psi$ the virtual support function of $L$. Then $|g_{\trop}+\psi|$ is bounded on $N_\R$. Since a toric line bundle admits a semipositive metric if and only if it is generated by global sections \cite[Corollary 4.8.5]{BPS}, the function $\psi$ is the support function of the lattice polytope
    \[
    \Delta_\psi\coloneqq\{ x\in M_\R:\langle x,u\rangle\geq \psi(u)\text{ for all }u\in N_\R \}
    \]
    by \cite[Section 3.4]{BPS}.
    Note that the value group $\Gamma$ on $K$ is dense, since $K$ is algebraically closed.
    Applying \cite[Proposition 2.5.24]{BPS} to the case $\Delta=\Delta_\psi$, we can find a sequence of convex $\Gamma$-rational piecewise affine functions $(g_i)_{i\geq 1}$ such that $|g_i+\psi|$ is bounded and $g_i$ converges uniformly to $g_{\trop}$ on $N_\R$. Hence, for each $i$, there exists a semipositive toric model metric $\lVert\cdot\rVert_i$ such that its tropical Green function is $g_i$. Then $\lVert\cdot\rVert$ is the uniform limit of $(\lVert\cdot\rVert_i)_{i\geq 1}$. The desired result follows.
\end{proof}

Thus, we can first treat the case where $\lVert\cdot\rVert$ is a semipositive toric model metric. In this case, it follows from \cite[Theorem 10.5]{firstDeltaForm} that $c_1(L|_X,\lVert\cdot\rVert)^{\wedge d}$ can be given by a wedge product of $\delta$-forms, as explained below.

We will denote the restriction of the tropicalization map to the torus,
\[
T^{\an}\to N_\R,
\]
by $\trop_T$. Its restriction to $(X\cap T)^{\an}$, i.e. the composition
\[
(X\cap T)^{\an}\xhookrightarrow{i^{\an}}T^{\an}\to N_\R,
\]
will be denoted by $\trop_{X\cap T}$.

It follows from \cite[7.7 and Proposition 9.15]{firstDeltaForm} that we have identities of $\delta$-forms
\[
c_1(L,\lVert\cdot\rVert)|_{T^{\an}}=c_1(L|_T,\lVert\cdot\rVert)=\ddc (g_{\trop}\circ\trop_T)
\]
and
\[
c_1(L,\lVert\cdot\rVert)|_{(X\cap T)^{\an}}=c_1(L|_{X\cap T},\lVert\cdot\rVert)=\ddc (g_{\trop}\circ\trop_{X\cap T}).
\]
Thus, using \Cref{lem:f star ddc}, we get the identities of $\delta$-forms
\begin{equation}\label{eq:id of first Chern form}
    c_1(L,\lVert\cdot\rVert)|_{T^{\an}}=c_1(L|_T,\lVert\cdot\rVert)=(\trop_T)^\star(\ddc g_{\trop})
\end{equation}
and
\begin{equation}\label{eq:id of first Chern form on X cap T}
    c_1(L,\lVert\cdot\rVert)|_{(X\cap T)^{\an}}=c_1(L|_{X\cap T},\lVert\cdot\rVert)=(\trop_{X\cap T})^{\star}(\ddc g_{\trop}).
\end{equation}

The projection formula together with \cite[Proposition 9.15]{firstDeltaForm} gives us an identity of $\delta$-currents
\begin{equation}\label{eq:id of MA of subvar}
    i_*\big(c_1(L|_X,\lVert\cdot\rVert)^{\wedge d}\big)=c_1(L,\lVert\cdot\rVert)^{\wedge d}\wedge\delta_X,
\end{equation}
where $c_1(L,\lVert\cdot\rVert)$ is the first-Chern $\delta$-form on $Y$ and $\delta_X$ is the current of integration on $X^{\an}$. By \cite[Corollary 1.4.5]{ToricMetricNonDiscrete}, the boundary $X^{\an}\backslash T^{\an}$ is a set of measure zero with respect to $c_1(L|_X,\lVert\cdot\rVert)^{\wedge d}$.

Recall that we have already introduced the tropical skeleton in \Cref{Section: tropical skeletons}. Since the $c_1(L|_X,\lVert\cdot\rVert)^{\wedge d}$-measure of $X^{\an}\backslash T^{\an}$ is zero, it is sufficient to study the measure $(c_1(L|_X,\lVert\cdot\rVert)^{\wedge d})|_{X^{\an}\cap T^{\an}}$. The next proposition shows that this measure is supported on the tropical skeleton $\Sigma(X^{\an}\cap T^{\an},i^{\an}|_{X^{\an}\cap T^{\an}})$.

\begin{prop}\label{prop:MA support}
    Let $\lVert\cdot\rVert$ be a semipositive toric metric on $L$. Then
    \[
    \operatorname{supp}\big( c_1(L|_X,\lVert\cdot\rVert)^{\wedge d} \big|_{X^{\an}\cap T^{\an}}\big)\subset\Sigma(X^{\an}\cap T^{\an},i^{\an}|_{X^{\an}\cap T^{\an}}).
    \]
\end{prop}
\begin{proof}
    Let $g_{\trop}:N_\R\to\R$ be the tropical Green function corresponding to $\lVert\cdot\rVert$. Since every semipositive toric metric $\lVert\cdot\rVert$ is a limit of semipositive toric model metrics by \Cref{lem:approximated semipos metric by model metrics}, we  first consider the case where $\lVert\cdot\rVert$ is a toric model metric. In this case, the tropical Green function $g_{\trop}$ is a convex $\Gamma$-rational piecewise affine function. Since $(\trop_{X\cap T})^\star(\ddc g_{\trop})^{\wedge d}\in B^{d,d}(X^{\an}\cap T^{\an})$ is a $\delta$-form on the purely $d$-dimensional analytic space $(X\cap T)^{\an}$, by \cite[Corollary 4.9(2)]{DeltaLubinTate}, we have that
    \[
    \operatorname{supp}\big( (\trop_{X\cap T})^\star(\ddc g_{\trop})^{\wedge d} \big)\subset\Sigma(X^{\an}\cap T^{\an},i^{\an}|_{X^{\an}\cap T^{\an}}).
    \]
    By (\ref{eq:id of first Chern form on X cap T}), we have
    \[
    c_1(L|_{X\cap T},\lVert\cdot\rVert)^{\wedge d}=(\trop_{X\cap T})^\star(\ddc g_{\trop})^{\wedge d}
    \]
    and thus
    \[
    \operatorname{supp}\big( c_1(L|_X,\lVert\cdot\rVert)^{\wedge d} \big|_{X^{\an}\cap T^{\an}}\big)\subset\Sigma(X^{\an}\cap T^{\an},i^{\an}|_{X^{\an}\cap T^{\an}}).
    \]

    In general, let $\lVert\cdot\rVert$ be a semipositive toric metric. Then as a consequence of \Cref{lem:approximated semipos metric by model metrics}, there exists a sequence of semipositive toric model metrics $(\lVert\cdot\rVert_i)_{i\geq 1}$ such that $c_1(L|_X,\lVert\cdot\rVert)^{\wedge d}$ is the limit of $( c_1(L|_X,\lVert\cdot\rVert_i)^d )_{i\geq 1}$ by weak convergence. Since $\Sigma(X^{\an}\cap T^{\an},i^{\an}|_{X^{\an}\cap T^{\an}})$ is a closed subset, the desired result follows from the conclusion for $\lVert\cdot\rVert_i$.
\end{proof}

\section{Polyhedral Monge-Amp{\`e}re measures}\label{Section: poly MA}

In \Cref{Section: tropicalized MA} and \Cref{Section: main thm}, we will describe non-archimedean Monge-Amp{\`e}re measures in terms of polyhedral Monge-Amp{\`e}re measures, thereby allowing the former to be computed locally on tropical skeletons. 

The theory of polyhedral Monge-Amp{\`e}re measures developed in \cite{MApoly} is formulated for $(\Z,\Q)$-polyhedral spaces and rational piecewise affine functions. In our setting, however, the natural polyhedral structures are $(\Z,\Gamma)$-polyhedral, and the piecewise affine functions arising from toric model metrics are $\Gamma$-rational. We therefore establish, in the form needed here, the corresponding version of the relevant results for $(\Z,\R)$-polyhedral complexes and piecewise affine functions with rational slopes.

\begin{defn}
    Let $\widetilde{C}\subset N_\R$ be a balanced weighted $(\Z,\R)$-polyhedral complex of pure dimension $d$, with nonnegative weights on its $d$-dimensional faces. Let $f:N_\R\to\R$ be a piecewise affine function. Following the tropical intersection theory developed by Mihatsch \cite{MihatschTIT}, which applies to arbitrary polyhedra and piecewise affine functions, we define the associated polyhedral Monge-Amp{\`e}re measure by
    \[
    \operatorname{MA}_{\operatorname{poly},\widetilde{C}}(f)=(\ddc f)^{\wedge d}\wedge\delta_{\widetilde{C}},
    \]
    where $\delta_{\widetilde{C}}$ is the integration current along the balanced polyhedral complex $\widetilde{C}$. If $f$ is convex, then $\operatorname{MA}_{\operatorname{poly},\widetilde{C}}(f)$ is a positive Radon measure by the positivity theory of $\delta$-forms; see \cite[Example 1.5 and Proposition 1.8]{GKpositivity}. That is the analogue, in our $(\Z,\R)$-polyhedral setting, of \cite[Definition 4.4]{MApoly}.
\end{defn}

Botero, Mazzon, and Pille-Schneider extend their tropical intersection-theoretic construction to \textit{PC-regularizable} functions in \cite[Section 4.3]{MApoly}. More precisely, for a fixed rational piecewise affine function $\gamma$, a \textit{PC-regularizable function} with respect to $\gamma$ can be characterized as a uniform limit of piecewise affine polyhedrally convex functions $f_i$ such that $|f_i-\gamma|$ is bounded; see \cite[Proposition 3.22 and Definition 3.23]{MApoly}. In particular, the convex piecewise affine functions with rational slopes considered here are PC-regularizable, even if their affine constants are not rational.

For a PC-regularizable function $f:N_\R\to\R$, we write
\[
\operatorname{MA}_{\operatorname{poly},\widetilde{C}}^{\operatorname{BMP}}(f)
\]
for its polyhedral Monge-Amp{\`e}re measure on a balanced $(\Z,\Q)$-polyhedral complex $\widetilde{C}$ on $N_\R$ in the sense of \cite{MApoly}. The proposition below, \Cref{prop:BMP MA = our poly MA in PL case}, compares $\operatorname{MA}_{\operatorname{poly},\widetilde{C}}^{\operatorname{BMP}}$ with Mihatsch's tropical intersection product, showing that our piecewise affine definition above is compatible with the construction of \cite{MApoly}.

Before proving the proposition, we first explain a local reduction to the rational case. This is the main tool we use to pass from the rational setting of \cite{MApoly} to the setting considered here. We first note that the polyhedral Monge-Amp{\`e}re measure of \cite{MApoly} is local: by \cite[Corollary 4.18]{MApoly}, if two PC-regularizable functions agree on an open subset, then their polyhedral Monge-Amp{\`e}re measures also agree there.

On the other hand, after translation, every $(\Z,\R)$-polyhedral complex is locally identified with a rational polyhedral fan. Thus, locally, the setting considered here can be reduced to the rational setting of \cite{MApoly}. Since this reduction will be used repeatedly in the proofs below, we record it explicitly.

\paragraph{Local reduction to the rational case.}
Let $\widetilde{C}\subset N_\R$ be a balanced $(\Z,\R)$-polyhedral complex. For every $v\in|\widetilde{C}|$, let $\mathcal{F}_v\coloneqq\operatorname{Star}_v(\widetilde{C})$
 be the local star of $\widetilde{C}$ at $v$. We denote by
 \[
 t_v(y)\coloneqq v+y
 \]
the translation by $v$. Then there exists an open neighborhood $U_v\subset N_\R$ of $v$ such that
\[
\widetilde{C}|_{U_v}=(t_v)_*(\mathcal{F}_v|_{V_v}),
\]
where $V_v\coloneqq U_v-v$. Note that $\mathcal{F}_v$ is a rational generalized fan, where ``generalized" means that its cones are not necessarily strongly convex. Indeed, under the translation $x=v+y$, every defining inequality $\langle m,x\rangle\geq c$ which is an equality at $v$ becomes $\langle m,y\rangle\geq 0$, so $\mathcal{F}_v$ is rational.

We will also use the following consequence of this local viewpoint. Suppose that $\widetilde{C}$ and $\widetilde{C}'$ are two balanced rational polyhedral complexes in $N_\R$ such that $\widetilde{C}|_U=\widetilde{C}'|_U$ as weighted cycle on an open subset $U\subset N_\R$. Then for every PC-regularizable function $f$ in the sense of \cite{MApoly},
\[
\operatorname{MA}_{\operatorname{poly},\widetilde{C}}^{\operatorname{BMP}}(f)|_U=\operatorname{MA}_{\operatorname{poly},\widetilde{C}'}^{\operatorname{BMP}}(f)|_U.
\]
Indeed, for rational piecewise affine functions this follows from the locality of the tropical intersection product on open subsets \cite[Remark 2.37]{MApoly}, and the general case follows by approximation and the continuity in \cite[Theorem 4.13]{MApoly}. We will use this together with the locality with respect to the function given by \cite[Corollary 4.18]{MApoly}.

\begin{prop}\label{prop:BMP MA = our poly MA in PL case}
    Let $\widetilde{C}\subset N_\R$ be a balanced $(\Z,\Q)$-polyhedral complex of pure dimension $d$ with nonnegative weights. Let $f:N_\R\to\R$ be a convex piecewise affine function with rational slopes, whose affine constants are not necessarily rational. Assume that $|f-\gamma|$ is bounded for some rational piecewise affine function $\gamma$. Then $f|_{|\widetilde{C}|}$ is PC-regularizable in the sense of \cite{MApoly}, and
    \begin{equation}\label{eq: prop comparison of BMP MA and our poly MA}
        \operatorname{MA}_{\operatorname{poly},\widetilde{C}}^{\operatorname{BMP}}(f)=(\ddc f)^{\wedge d}\wedge\delta_{\widetilde{C}}.
    \end{equation}
\end{prop}
\begin{proof}
    Since $f$ is convex piecewise affine, we can write $f$ as
    \[
    f(x)=\max_{j\in J}\left\{ \langle a_j,x\rangle+b_j \right\},
    \]
    where $a_j\in M_\Q$, $b_j\in\R$ and $J$ is finite. We first show that $f|_{|\widetilde{C}|}$ is PC-regularizable. Since $\Q$ is dense, for every $j\in J$, we can choose a sequence $\{b_{j,k}\}_{k\geq 1}\subset\Q$ such that $b_{j,k}\to b_j$ as $k\to\infty$. Set
    \[
    f_k(x)\coloneqq\max_{j\in J}\left\{ \langle a_j,x\rangle+b_{j,k} \right\}
    \]
    for $k\geq 1$. Then each $f_k$ is a convex rational piecewise affine function and $f$ is the uniform limit of $f_k$. Since $|f-\gamma|$ is bounded, we have $|f_k-\gamma|$ is bounded for $k$ large enough. This shows that $f$ is PC-regularizable and thus $\operatorname{MA}_{\operatorname{poly},\widetilde{C}}^{\operatorname{BMP}}(f)$ is well-defined.

    We now prove the desired equality (\ref{eq: prop comparison of BMP MA and our poly MA}). We first locally identify $f$ with a rational piecewise affine function. In this rational case, the construction of \cite{MApoly} is given by the tropical intersection product, and thus agrees with our definition of polyhedral Monge-Amp{\`e}re measure. 

    To obtain this local identification, we use the local reduction to the rational case. Fix a point $v\in\widetilde{C}$ and take the local star $\mathcal{F}_v\coloneqq\operatorname{Star}_v(\widetilde{C})$, which is a rational generalized fan. As discussed above, $\mathcal{F}_v$ and $(t_{-v})_*(\widetilde{C})$ agree in an neighborhood of $0$.
    
    We now modify $f\circ t_v$ near $0$ such that it agrees there with a rational piecewise affine function. Recall that the polyhedral Monge-Ampère measure is invariant under addition of constants \cite[Theorem 4.13(3)]{MApoly}. Set
    \[
    f_0(y)\coloneqq f(v+y)-f(v)=(f\circ t_v)(y)-f(v).
    \]
    Then
    \[
    f_0(y)=\max_{j\in J}\left\{ \langle a_j,y\rangle+\beta_j \right\}\quad\text{for}\quad \beta_j=\langle a_j,v\rangle+b_j-f(v)\leq 0.
    \]
    Note that $\beta_j=0$ for at least one $j$, since $f(v)$ is the maximum of affine pieces at $v$. We now modify $f_0$ to obtain a rational piecewise affine function locally near $0$. For $j\in J$ with $\beta_j=0$, we set $\beta_j'=\beta_j=0$. For $j\in J$ with $\beta_j<0$, we replace $\beta_j$ by a negative rational number $\beta_j'\in\Q$. Set
    \[
    g_0(y)\coloneqq \max_{j\in J}\left\{ \langle a_j,y\rangle+\beta_j' \right\}.
    \]
    Then $f_0=g_0$ locally near $0$. Thus, the locality of polyhedral Monge-Amp{\`e}re measures from \cite[Corollary 4.18]{MApoly} gives us
    \[
    \operatorname{MA}_{\operatorname{poly},\mathcal{F}_v}^{\operatorname{BMP}}(f_0)=\operatorname{MA}_{\operatorname{poly},\mathcal{F}_v}^{\operatorname{BMP}}(g_0)
    \]
    in an open neighborhood of $0$. Since $g_0$ is rational, it follows from \cite[Definition 4.4]{MApoly} and \cite[Lemma 4.9]{MihatschTIT} that
    \[
    \operatorname{MA}_{\operatorname{poly},\mathcal{F}_v}^{\operatorname{BMP}}(g_0)=(\ddc g_0)^{\wedge d}\wedge\delta_{\mathcal{F}_v}. 
    \]
    Since Mihatsch's tropical intersection theory from \cite[Lemma 4.9]{MihatschTIT} is local and $f_0=g_0$ near $0$, we have $(\ddc f_0)^{\wedge d}\wedge\delta_{\mathcal{F}_v}=(\ddc g_0)^{\wedge d}\wedge\delta_{\mathcal{F}_v}$ and thus
    \begin{equation}\label{eq: eq1 in prop comparison BMP MA}
        \operatorname{MA}_{\operatorname{poly},\mathcal{F}_v}^{\operatorname{BMP}}(f_0)=(\ddc f_0)^{\wedge d}\wedge\delta_{\mathcal{F}_v}
    \end{equation}
    in an open neighborhood of $0$. It remains to transfer this equality (\ref{eq: eq1 in prop comparison BMP MA}) back to $\widetilde{C}$. In the right-hand side, we have immediately
    \[
    (t_{-v})_*\left( (\ddc f)^{\wedge d}\wedge\delta_{\widetilde{C}} \right)=(\ddc f_0)^{\wedge d}\wedge\delta_{\mathcal{F}_v}
    \]
    in an open neighborhood of $0$. 
    
    Therefore, it suffices to show that in the left-hand side of (\ref{eq: eq1 in prop comparison BMP MA}) we have
    \[
    (t_{-v})_*\left( \operatorname{MA}_{\operatorname{poly},\widetilde{C}}^{\operatorname{BMP}}(f) \right)=\operatorname{MA}_{\operatorname{poly},\mathcal{F}_v}^{\operatorname{BMP}}(f_0)
    \]
    in an open neighborhood of $0$. By the invariance of polyhedral Monge-Amp{\`e}re measures with respect to additive constants \cite[Theorem 4.13(3)]{MApoly}, it is equivalent to show that
    \begin{equation}\label{eq: eq2 in prop comparison BMP MA}
        (t_{-v})_*\left( \operatorname{MA}_{\operatorname{poly},\widetilde{C}}^{\operatorname{BMP}}(f) \right)=\operatorname{MA}_{\operatorname{poly},\mathcal{F}_v}^{\operatorname{BMP}}(f\circ t_v)
    \end{equation}
    
    Since $v$ need not be rational, we approximate it with rational points $v_k$ whose local stars $\operatorname{Star}_{v_k}(\widetilde{C})$ agree with $\mathcal{F}_v$. Indeed, let $\tau$ be the rational face of $\widetilde{C}$ such that $v\in\operatorname{relint}(\tau)$. We take rational points $\{v_k\}_{k\geq 1}$ such that $v_k\in\operatorname{relint}(\tau)\cap N_{\Q}$ and $v_k\to v$ as $k\to\infty$. Then $\operatorname{Star}_{v_k}(\widetilde{C})=\mathcal{F}_v$. 

    Since $v_k\in N_\Q$, translation by $v_k$ preserves $(\Z,\Q)$-polyhedral spaces and rational piecewise affine functions. Hence, by translation invariance in the piecewise affine case and the continuity \cite[Theorem 4.13]{MApoly}, we have
    \[
    (t_{-v_k})_*\left( \operatorname{MA}_{\operatorname{poly},\widetilde{C}}^{\operatorname{BMP}}(f) \right)= \operatorname{MA}_{\operatorname{poly},(t_{-v_k})_*(\widetilde{C})}^{\operatorname{BMP}}(f\circ t_{v_k}).
    \]
    For $k$ sufficiently large, $(t_{-v_k})_*(\widetilde{C})$ and $\mathcal{F}_v$ agree in a fixed open neighborhood of $0$. Using the locality discussed above, we have
    \[
    (t_{-v_k})_*\left( \operatorname{MA}_{\operatorname{poly},\widetilde{C}}^{\operatorname{BMP}}(f) \right)= \operatorname{MA}_{\operatorname{poly},\mathcal{F}_v}^{\operatorname{BMP}}(f\circ t_{v_k})
    \]
    in this neighborhood. 

    Since $\{f\circ t_{v_k}\}_{k\geq 1}$ is a sequence of convex piecewise affine functions with rational slopes, converging uniformly to $f\circ t_v$ as $k\to\infty$, by \cite[Theorem 4.13]{MApoly},
    \[
    \operatorname{MA}_{\operatorname{poly},\mathcal{F}_v}^{\operatorname{BMP}}(f\circ t_{v_k}) \to \operatorname{MA}_{\operatorname{poly},\mathcal{F}_v}^{\operatorname{BMP}}(f\circ t_v)
    \]
    weakly against compactly supported functions on $|\mathcal{F}_v|$ as $k\to\infty$. On the other hand, we have
    \[
    (t_{-v_k})_*\left( \operatorname{MA}_{\operatorname{poly},\widetilde{C}}^{\operatorname{BMP}}(f) \right)\to (t_{-v})_*\left( \operatorname{MA}_{\operatorname{poly},\widetilde{C}}^{\operatorname{BMP}}(f) \right)
    \]
    weakly against compactly supported functions on $|(t_{-v_k})_*(\widetilde{C})|$ as $k\to\infty$. Comparing the two weak limits above yields (\ref{eq: eq2 in prop comparison BMP MA}), which proves the desired result (\ref{eq: prop comparison of BMP MA and our poly MA}).
\end{proof}

In our setting, we will use the weak continuity established in \Cref{Prop:MA poly weak convergence} to pass from semipositive toric model metrics to arbitrary semipositive toric metrics by uniform approximation. Moreover, the approximating functions arising from toric model metrics are convex $\Gamma$-rational piecewise affine functions on $N_\R$. We therefore need an extension of the weak continuity in \cite[Theorem 4.13]{MApoly} beyond the rational setting, which we formulate below.

\begin{prop}\label{Prop:MA poly weak convergence}
    Let $\widetilde{C}\subset N_\R$ be a balanced weighted $(\Z,\R)$-polyhedral complex of pure dimension $d$ with nonnegative weights, and let $\gamma:N_\R\to\R$ be a convex piecewise affine function with rational slopes. Let $(f_i)_{i\geq 1}$ be a sequence of convex piecewise affine functions with rational slopes on $N_\R$ such that $|f_i-\gamma|$ is bounded for every $i$, and suppose that $f_i$ converges uniformly to a convex function $f:N_\R\to\R$. Then the measures
    \[
    \operatorname{MA}_{\operatorname{poly},\widetilde{C}}(f_i)
    \]
    converge weakly against compactly supported continuous functions on $|\widetilde{C}|$. The limit is a positive Radon measure on $|\widetilde{C}|$ and is independent of the choice of such an approximating sequence.
\end{prop}
For such a function $f$, we define
\[
\operatorname{MA}_{\operatorname{poly},\widetilde{C}}(f)\coloneqq \lim\limits_{i\to\infty}\operatorname{MA}_{\operatorname{poly},\widetilde{C}}(f_i),
\]
where the limit is taken in the sense of weak convergence against compactly supported continuous functions on $|\widetilde{C}|$.

We now prove \Cref{Prop:MA poly weak convergence} using local reduction to the rational case.

\begin{proof}
    Fix a point $v\in|\widetilde{C}|$. Let $t_v$ be the translation by $v$ and $\mathcal{F}_v$ be the local star of $\widetilde{C}$ at $v$. As we have discussed above in the local reduction to the rational case, $\mathcal{F}_v$ and $(t_{-v})_*(\widetilde{C})$ agree in a fixed open neighborhood $U_0$ of $0$.

    Since $\Q$ is dense, we can find a convex rational piecewise affine function $\eta_v$ such that $|\eta_v-\gamma\circ t_v|$ is bounded. This implies that $|f_i\circ t_v-\eta_v|$ and $|f\circ t_v-\eta_v|$ are bounded. Using essentially the same argument in the proof of \Cref{prop:BMP MA = our poly MA in PL case}, we obtain that $f_i\circ t_v$ and $f\circ t_v$ can be approximated by convex rational piecewise affine functions. Therefore, $(f_i\circ t_v)|_{|\mathcal{F}_v|}$ and $(f\circ t_v)|_{|\mathcal{F}_v|}$ are PC-regularizable in the sense of \cite{MApoly} with respect to the rational piecewise affine function $\eta_v$.

    Choose a rational simplicial refinement of $\mathcal{F}_v$ on which $\eta_v$ is affine. Since $f\circ t_v$ is the uniform limit of $f_i\circ t_v$, it follows from \cite[Theorem 4.13]{MApoly} that
    \[
    \operatorname{MA}_{\operatorname{poly},\mathcal{F}_v}^{\operatorname{BMP}}(f_i\circ t_v)\longrightarrow \operatorname{MA}_{\operatorname{poly},\mathcal{F}_v}^{\operatorname{BMP}}(f\circ t_v)
    \]
    weakly against the compactly supported continuous functions on $|\mathcal{F}_v|$. By \Cref{prop:BMP MA = our poly MA in PL case} and the locality of tropical intersection, we have
    \[
    \operatorname{MA}_{\operatorname{poly},\widetilde{C}}(f_i)=(t_v)_*\left( \operatorname{MA}_{\operatorname{poly},\mathcal{F}_v}^{\operatorname{BMP}}(f_i\circ t_v) \right)
    \]
    in a fixed open neighborhood $U_v=(t_v)_*(U_0)$ of $v$. Therefore,
    \[
    \operatorname{MA}_{\operatorname{poly},\widetilde{C}}(f_i)\longrightarrow (t_v)_*\left( \operatorname{MA}_{\operatorname{poly},\mathcal{F}_v}^{\operatorname{BMP}}(f\circ t_v) \right)
    \]
    in the open neighborhood $U_v$ of $v$. The local limits agree on overlaps, since they are limits of the same piecewise affine measures there. They therefore define a positive Radon measure $\operatorname{MA}_{\operatorname{poly},\widetilde{C}}(f)$ on $|\widetilde{C}|$.

    Finally, let $\phi$ be a compactly supported continuous function on $|\widetilde{C}|$. Then the support of $\phi$ can be covered by finitely many such open neighborhoods $U_v$. A partition of unity and local convergence give us
    \[
    \int \phi\, \operatorname{d} \operatorname{MA}_{\operatorname{poly},\widetilde{C}}(f_i)\longrightarrow \int \phi\, \operatorname{d} \operatorname{MA}_{\operatorname{poly},\widetilde{C}}(f).
    \]
    The independence of the approximating sequence follows immediately, since the local limit is given by $\operatorname{MA}_{\operatorname{poly},\mathcal{F}_v}^{\operatorname{BMP}}(f\circ t_v)$, which depends only on $f$.
\end{proof}

We next establish the locality of this construction, which will be needed in \Cref{Section: main thm} when we apply polyhedral Monge-Amp{\`e}re measures to weighted open representatives of tropical varieties of analytic germs.

\begin{defn}
    Let $\widetilde{C}\subset N_\R$ and $f:N_\R\to\R$ be as in \Cref{Prop:MA poly weak convergence}. Let $\Omega\subset N_\R$ be an open subset. If $C\coloneqq\widetilde{C}|_\Omega$, then we set
    \[
    \operatorname{MA}_{\operatorname{poly},C}(f)\coloneqq\operatorname{MA}_{\operatorname{poly},\widetilde{C}}(f)\big|_{\Omega}.
    \]
\end{defn}

We need to check that this local polyhedral Monge-Amp{\`e}re measure is well-defined. Namely, if $\widetilde{C}'$ is another balanced $(\Z,\R)$-polyhedral complex with nonnegative weights such that $\widetilde{C}'|_\Omega=\widetilde{C}|_\Omega$ as tropical cycle, then we need to show that
\begin{equation}\label{eq: locality of MA poly}
    \operatorname{MA}_{\operatorname{poly},\widetilde{C}'}(f)\big|_{\Omega}=\operatorname{MA}_{\operatorname{poly},\widetilde{C}}(f)\big|_{\Omega}.
\end{equation}
Indeed, let $(f_i)_{i\geq 1}$ be the sequence of convex piecewise affine functions with rational slopes converging uniformly to $f$ as in \Cref{Prop:MA poly weak convergence}. The locality of tropical intersection product gives
\[
\operatorname{MA}_{\operatorname{poly},\widetilde{C}'}(f_i)\big|_{\Omega}=\operatorname{MA}_{\operatorname{poly},\widetilde{C}}(f_i)\big|_{\Omega};
\]
see \cite[Lemma 4.9]{MihatschTIT} and the local formalism of polyhedral currents in \cite[Section 2]{DeltaLubinTate}. Passing to the weak limit by \Cref{Prop:MA poly weak convergence} yields (\ref{eq: locality of MA poly}). Hence, the definition is independent of the choice of $\widetilde{C}$.

\paragraph{The tropical Green functions in our setting.}
We now verify that the tropical Green functions associated with toric metrics arising in our setting satisfy the assumptions of \Cref{Prop:MA poly weak convergence}. Let $Y$ be a toric variety with dense torus $T$ and let $L$ be a toric line bundle. Let $g_{\trop}:N_\R\to\R$ be the tropical Green function associated to a semipositive toric metric on $L$.

Let $\psi:N_\R\to\R$ be the piecewise linear virtual support function corresponding to $L$. Then $|g_{\trop}+\psi|$ is bounded. By \Cref{lem:approximated semipos metric by model metrics}, there exists a sequence of convex $\Gamma$-rational piecewise affine functions $(g_i)_{i\geq 1}$ such that $|g_i+\psi|$ is bounded and $(g_i)_{i\geq 1}$ converges uniformly to $g_{\trop}$ on $N_\R$. We get the desired result by taking $\gamma\coloneqq-\psi$.

\paragraph{Comparison with delta-forms and Bedford-Taylor products.} 
We conclude this section by identifying polyhedral Monge-Amp{\`e}re measures with the corresponding products of $\delta$-forms for piecewise smooth functions. This identification will be used later. 

\begin{prop}\label{prop: PS comparison of poly MA and delta form}
    Let $\Omega\subset N_\R$ be an open subset, and let $C$ be the restriction to $\Omega$ of a balanced $(\Z,\R)$-polyhedral complex in $N_\R$ of pure dimension $d$ with nonnegative weights. Let $f:N_\R\to\R$ be a convex piecewise smooth function satisfying the approximation hypothesis of \Cref{Prop:MA poly weak convergence}. Then
    \[
    \operatorname{MA}_{\operatorname{poly},C}(f)=(\ddc f)^{\wedge d}\wedge\delta_C.
    \]
\end{prop}
\begin{proof}
    We use the Bedford-Taylor product defined in \cite[Section 4.1]{BGJK21} as a bridge. To distinguish it from the product of $\delta$-forms, we write $\wedge_{\operatorname{BT}}$ for the Bedford-Taylor product. Since $C$ has nonnegative weights, the current $\delta_C$ is positive. Since $C$ is balanced, the current $\delta_C$ is closed by \cite[Proposition 3.8]{Gub16DeltaC}. Thus $\delta_C$ is a positive closed Lagerberg current in the sense of \cite[Section 4.1]{BGJK21}.

    Let $(f_i)_{i\geq 1}$ be the sequence of convex piecewise affine functions with rational slopes converging uniformly to $f$ as in the statement of \Cref{Prop:MA poly weak convergence}. By \cite[Theorem 4.1.2]{BGJK21}, 
    \[
    (\ddc f_i)^{\wedge_{\operatorname{BT}}d}\wedge_{\operatorname{BT}}\delta_C\longrightarrow (\ddc f)^{\wedge_{\operatorname{BT}} d}\wedge_{\operatorname{BT}}\delta_C
    \]
    weakly as Lagerberg currents.

    We will compare the Bedford-Taylor product with the polyhedral Monge-Amp{\`e}re measure and the product of $\delta$-forms occurring in our setting. More precisely, we are going to show that
    \begin{equation}\label{eq: comparison BT and delta form}
        (\ddc f)^{\wedge_{\operatorname{BT}} d}\wedge_{\operatorname{BT}}\delta_C = (\ddc f)^{\wedge d}\wedge\delta_C
    \end{equation}
    and
    \begin{equation}\label{eq: comparison poly MA and BT}
        (\ddc f)^{\wedge_{\operatorname{BT}} d}\wedge_{\operatorname{BT}}\delta_C=\operatorname{MA}_{\operatorname{poly},C}(f).
    \end{equation}
    Let us first prove (\ref{eq: comparison BT and delta form}). Let $h:N_\R\to\R$ be a convex piecewise smooth function and let $S$ be a positive closed polyhedral current. By the proof of \cite[Theorem 4.1.2]{BGJK21}, the Bedford-Taylor product is defined inductively following \cite[Section III.3]{Dem12}, and thus
    \[
    (\ddc h)\wedge_{\operatorname{BT}} S=\ddc (hS).
    \]
    On the other hand, by the Leibniz rule for $\delta$-forms,
    \[
    (\ddc h)\wedge S=\ddc(h S)-h\ddc S+d''h\wedge d'S-d'h\wedge d''S.
    \]
    Since every $S$ is closed, we have $d'S=d''S=0$ and thus
    \[
    (\ddc h)\wedge S=\ddc (hS).
    \]
    This proves (\ref{eq: comparison BT and delta form}) inductively. It remains to show (\ref{eq: comparison poly MA and BT}). Note that for the piecewise affine function $f_i:N_\R\to\R$, we have
    \[
    \operatorname{MA}_{\operatorname{poly},C}(f_i)=(\ddc f_i)^{\wedge d}\wedge\delta_C.
    \]
    The left-hand side converges weakly to $\operatorname{MA}_{\operatorname{poly},C}(f)$ by \Cref{Prop:MA poly weak convergence}, while the right-hand side converges weakly to $(\ddc f)^{\wedge_{\operatorname{BT}} d}\wedge_{\operatorname{BT}}\delta_C$ by \cite[Theorem 4.1.2(ii)]{BGJK21}, which implies the desired result.
\end{proof}

\begin{rmk}\label{prop: comparison of poly MA and BT product}
    The proof of \Cref{prop: PS comparison of poly MA and delta form} also shows that the comparison with the Bedford-Taylor product does not require the piecewise smoothness assumption. More precisely, for the class of functions considered in \Cref{Prop:MA poly weak convergence}, the polyhedral Monge-Amp{\`e}re measure
    \[
    \operatorname{MA}_{\operatorname{poly},C}(f)
    \]
    agrees, as a Lagerberg current, with the Bedford-Taylor product
    \[
    (\ddc f)^{\wedge_{\operatorname{BT}}d}\wedge_{\operatorname{BT}}\delta_C
    \]
    defined in \cite[Section 4.1]{BGJK21}. The Bedford-Taylor product extends the wedge product of $\ddc$-currents to nonsmooth plurisubharmonic functions and is a basic tool in the tropical approach of \cite{BGJK21} to non-archimedean pluripotential theory and Monge-Amp{\`e}re equations.
\end{rmk}

\section{Tropicalized Monge-Amp{\`e}re measures}\label{Section: tropicalized MA}

We keep the notation introduced in \Cref{Section: NA MA measure as delta current}: let $Y=Y_\Sigma$ be a proper toric variety over $K$ of dimension $n$ with dense torus $T$ and let $X$ be a closed subvariety of $Y$ of dimension $d$ with embedding $i:X\hookrightarrow Y$ such that $X\cap T\neq\varnothing$. Let $L=\mathcal{O}(D)$ be a toric line bundle on $Y$ and let $\lVert\cdot\rVert$ be a semipositive toric metric on $L$ with tropical Green function $g_{\trop}:N_\R\to\R$. We denote the restrictions of the tropicalization map $\trop:Y_\Sigma^{\an}\to N_\Sigma$ to $T^{\an}$ and $X^{\an}$ by
\[
\trop_T:T^{\an}\to N_\R
\]
and
\[
\trop_X:X^{\an}\hookrightarrow Y^{\an}\to N_\Sigma,
\]
respectively. 

In this section, we are going to study the tropicalized Monge-Amp{\`e}re measures $(\trop_X)_*(c_1(L|_X,\lVert\cdot\rVert)^{\wedge d})$. As in \Cref{Section: NA MA measure as delta current}, for the Monge-Amp{\`e}re measure $c_1(L|_X,\lVert\cdot\rVert)^{\wedge d}$, the boundary $X\backslash T$ is a set of measure zero. Therefore, for the tropicalized measure $(\trop_X)_*(c_1(L|_X,\lVert\cdot\rVert)^{\wedge d})$, the boundary $N_\Sigma\backslash N_\R$ is a set of measure zero. Thus, it is sufficient for us to study $( (\trop_X)_*(c_1(L|_X,\lVert\cdot\rVert)^{\wedge d}) )|_{N_\R}$. 

Since $X\cap T$ is a closed subvariety of $T$, by the Bieri-Groves theorem and the work of Speyer-Sturmfels, see for instance \cite[Section 2.3]{tropicalzedVar}, we have that $\trop(X^{\an}\cap T^{\an})$ is a tropical variety, i.e. it is a polyhedral complex purely of dimension $d$ where every polyhedron $\sigma$ with dimension $d$ has nonnegative weight $m_\sigma$ satisfying the balancing condition. 

We first suppose that $\lVert\cdot\rVert$ is a toric model metric. Then $g_{\trop}$ is a convex $\Gamma$-rational piecewise affine function and $c_1(L|_X,\lVert\cdot\rVert)$ is a $\delta$-current, as explained in \Cref{Section: NA MA measure as delta current}. By (\ref{eq:id of first Chern form}) and (\ref{eq:id of MA of subvar}), we have
\[
\big( i_*\big(c_1(L|_X,\lVert\cdot\rVert)^{\wedge d}\big) \big)\big|_{T^{\an}}=(\trop_T)^\star(\ddc g_{\trop})^{\wedge d}\wedge\delta_{X\cap T},
\]
where $\delta_{X\cap T}$ is the current of integration on $X^{\an}\cap T^{\an}$. Then the compatibility of integration with tropicalization given in \cite[Example 7.2.5]{newArxivGGK} implies, in our setting, the following equality of currents:
\begin{equation}\label{lem:projection formula}
    \big( (\trop_X)_*(c_1(L|_X,\lVert\cdot\rVert)^{\wedge d}) \big)\big|_{N_\R}=(\ddc g_{\trop})^{\wedge d}\wedge\delta_{\trop(X^{\an}\cap T^{\an})},
\end{equation}
where $\delta_{\trop(X^{\an}\cap T^{\an})}$ is the current of integration along the weighted tropical cycle $\trop(X^{\an}\cap T^{\an})$.

The following comparison between tropical intersection products and real Monge-Amp{\`e}re measures is a standard consequence of known results. A closely related statement in the rational setting appears in \cite[Proposition 4.12]{MApoly}; the formulation below can be viewed as its $\Gamma$-rational version. We include a direct proof in the language of $\delta$-forms and $\delta$-currents to make our normalization explicit.

\begin{prop}\label{prop:ddf and MA on Rn}
    Let $f:N_\R\to\R$ be a convex $\Gamma$-rational piecewise affine function. Then
    \[
    (\ddc f)^{\wedge n}=n!\operatorname{MA}_\R(f).
    \]
\end{prop}
\begin{proof}
    Suppose $f$ is piecewise affine on a polyhedral complex $\Sigma$ of dimension $n$. It follows from \cite[Proposition 2.7.4]{BPS} that
    \[
    \operatorname{MA}_\R(f)=\sum_{v}\operatorname{Vol}(\partial f(v))\delta_v,
    \]
    where $v$ ranges over all vertices of $\Sigma$. Using tropical intersection theory, see \cite[Lemma 4.9]{MihatschTIT} or \cite{AllermannRau} for details, we can deduce that $(\ddc f)^{\wedge n}$ is the sum of Dirac measures supported on the vertices of $\Sigma$ and is invariant under a translation of $f$ by an affine function. 

    Therefore, the desired result is a local property, i.e. it is sufficient to show that if $\Sigma$ is a fan and $f(0)=0$, then we have
    \[
    (\ddc f)^{\wedge n}(\{0\})= n! \operatorname{Vol}(\partial f(0)).
    \]
    Let $\Sigma_n$ be the set of maximal cones in $\Sigma$. Then $f(x)=\operatorname{max}_{\sigma\in\Sigma_n}\langle a_\sigma,x\rangle$ for some $a_\sigma$. In this case $f$ is the support function of the convex hull $\operatorname{conv}(a_\sigma:\sigma\in\Sigma_n)$ and it follows from \cite[Theorem 8.24]{Rockafellar} that
    \[
    \partial f(0)=\operatorname{conv}(a_\sigma:\sigma\in\Sigma_n).
    \]
    On the other hand, by \cite[Proposition 5.12]{Lagerberg} we have
    \[
    (\ddc f)^{\wedge n}(\{0\})=n!\operatorname{Vol}(\operatorname{conv}(a_\sigma:\sigma\in\Sigma_n)),
    \]
    which proves the proposition.
\end{proof}

This ambient identity also appears in \cite[Remark 6.2.12]{BGJK21} with a different normalization of the real Monge-Amp{\`e}re measure. We use the normalization introduced in \ref{convention:normalization of MA}. The following lemma follows immediately from \Cref{prop:ddf and MA on Rn}.

\begin{lem}\label{lem:ddf and MA}
    Let $f:N_\R\to\R$ be a convex $\Gamma$-rational piecewise affine function and let $C$ be a tropical cycle of dimension $d\leq n$. Then
    \[
    \big((\ddc f)^{\wedge d}\wedge\delta_C\big)\big|_\sigma = d!\, m_{\sigma}\, \operatorname{MA}_\R\!\left(f|_\sigma\right)
    \]
    for every maximal open face $\sigma$ of $C$ with weight $m_\sigma$.
\end{lem}

\begin{ex}
    Let $Y_\Sigma=\mathbb{P}^2_K$ where $\Sigma$ is the complete fan given by the standard basis $e_1=(1,0),e_2=(0,1)$ and the vector $e_0\coloneqq-e_1-e_2$. Let $f=x_0+x_1+x_2\in K[x_0,x_1,x_2]$ and $X=V(f)$. In the standard torus of $\mathbb{P}^2_K$ let us consider $\bar{f}=x_1+x_2+1$ and the hypersurface $V(\bar{f})$. Note that
    \[
    \trop(\bar{f})=\operatorname{min}\{0,x_1,x_2\}.
    \]
    Kapranov's Theorem \cite[Theorem 3.1.3]{IntroTropical} and \cite[Proposition 3.7]{guide} imply that
    \[
    \trop(V(\bar{f})^{\an})=\{x\in\R^2:\text{the minimum in }\trop(\bar{f})\text{ is achieved at least twice}\}.
    \]
    Hence, the tropical variety $\trop(V(\bar{f})^{\an})$ is the collection of the $1$-dimensional rays $\tau_1=\operatorname{cone}((1,0)),\ \tau_2=\operatorname{cone}((0,1)),\ \tau_0=\operatorname{cone}((-1,-1))$ and the vertex $0$. Since the regular subdivision of the Newton polytope of $\bar{f}$ w.r.t. $(0,0,0)$ is trivial, by \cite[Lemma 3.4.6]{IntroTropical} the multiplicities of all the $3$ cones above are $1$.

    Let $L=O(1)$ and let $\lVert\cdot\rVert$ be the toric metric associated to
    \[
    g_{\trop}(x)=\operatorname{max}\left\{0,-x_1,-x_2,-\frac{1}{2}x_1-\frac{1}{2}x_2+1\right\}.
    \]
    Then $g_{\trop}$ is convex piecewise affine on the polyhedral complex with maximal polyhedrons
    \begin{align*}
        &\sigma_1=\{(x_1,x_2):x_1\leq 0,\ x_1-x_2+2\leq 0\},\\
        &\sigma_2=\{(x_1,x_2):x_1\geq 0,\ x_1+x_2-2\geq 0,\ x_2\geq 0\},\\
        &\sigma_3=\{(x_1,x_2):x_1-x_2+2\geq 0,\ x_1-x_2-2\leq 0,\ x_1+x_2-2\leq 0\},\\
        &\sigma_4=\{(x_1,x_2):x_1-x_2-2\geq 0,\ x_2\leq 0\}.
    \end{align*}
    This polyhedral complex is illustrated in \Cref{fig:polyhedral complex in example}.
    
    \begin{figure}
        \centering
           \begin{tikzpicture}[
    scale=0.67,
    transform shape,
    x=0.9cm,
    y=0.9cm,
    >=Latex,
    line cap=round,
    line join=round
]
    \colorlet{segcol}{blue!70!black}
    \colorlet{funccol}{red!75!black}

    \fill[blue!7]
        (-5,-3) -- (-5,5) -- (0,5) -- (0,2) -- cycle;

    \fill[red!7]
        (0,2) -- (0,5) -- (5,5) -- (5,0) -- (2,0) -- cycle;

    \fill[green!9]
        (-5,-5) -- (-5,-3) -- (0,2) -- (2,0) -- (-3,-5)
        -- cycle;

    \fill[orange!9]
        (-3,-5) -- (5,-5) -- (5,0) -- (2,0) -- cycle;

    \draw[very thick,->] (-5,0) -- (5,0)
        node[below right] {$x_1$};

    \draw[very thick,->] (0,-5) -- (0,5)
        node[above left] {$x_2$};

    \node[below left] at (0,0) {$O$};

    \foreach \x in {-2,2}
        \draw[very thick] (\x,0.10) -- (\x,-0.10)
            node[below=2pt] {$\x$};

    \foreach \y in {-2,2}
        \draw[very thick] (0.10,\y) -- (-0.10,\y)
            node[left=2pt] {$\y$};

    \draw[very thick,segcol] (-5,-3) -- (0,2);
    \draw[very thick,segcol] (0,2) -- (0,5);
    \draw[very thick,segcol] (0,2) -- (2,0);
    \draw[very thick,segcol] (2,0) -- (5,0);
    \draw[very thick,segcol] (-3,-5) -- (2,0);

    \node[font=\large,text=funccol] at (-2.5,3.1)
        {$\sigma_1$};

    \node[font=\large,text=funccol] at (2.7,3.0)
        {$\sigma_2$};

    \node[font=\large,text=funccol] at (-1.8,-1.8)
        {$\sigma_3$};

    \node[font=\large,text=funccol] at (2.7,-2.6)
        {$\sigma_4$};
\end{tikzpicture}

        \caption{The polyhedral decomposition $\sigma_1,\sigma_2,\sigma_3,\sigma_4$ of $\R^2$}
        \label{fig:polyhedral complex in example}
    \end{figure}
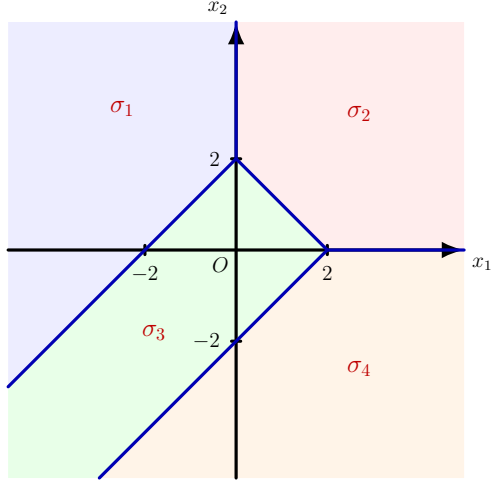
    
    Note that
    \[
    g_{\trop}|_{\sigma_1}=-x_1,\quad g_{\trop}|_{\sigma_2}=0,\quad g_{\trop}|_{\sigma_3}=-\frac{1}{2}x_1-\frac{1}{2}x_2+1,\quad g_{\trop}|_{\sigma_4}=-x_2.
    \]
    Let $\tau_i^\circ$ be the relative interior of $\tau_i$.
    By the tropical intersection theory we have
    \begin{align*}
        &\big((\ddc g_{\trop})\wedge\delta_{\trop(X^{\an}\cap T^{\an})}\big)\big|_{\tau_1^\circ}=\frac{1}{2}\delta_{(2,0)},\\
        &\big((\ddc g_{\trop})\wedge\delta_{\trop(X^{\an}\cap T^{\an})}\big)\big|_{\tau_2^\circ}=\frac{1}{2}\delta_{(0,2)},\\
        &\big((\ddc g_{\trop})\wedge\delta_{\trop(X^{\an}\cap T^{\an})}\big)\big|_{\tau_0^\circ}=0.
    \end{align*}
    By computing the volume of the sub-differential of $g_{\trop}$, we get
    \[
    \operatorname{MA}_\R\left(g_{\trop}|_{\tau_1^\circ}\right)=\frac{1}{2}\delta_{(2,0)},\quad \operatorname{MA}_\R\left(g_{\trop}|_{\tau_2^\circ}\right)=\frac{1}{2}\delta_{(0,2)},\quad\operatorname{MA}_\R\left(g_{\trop}|_{\tau_0^\circ}\right)=0.
    \]
\end{ex}

Since we are going to compute the tropicalized Monge-Amp{\`e}re measure for any semipositive toric metric corresponding to a convex tropical Green function, we need a generalization of \Cref{lem:ddf and MA} for convex functions. For this purpose, we give the definition of $(\ddc f)^d$ for a convex function $f$ following \cite[Section 2]{Lagerberg}.

It is well-known that given a convex function $f:N_\R\to\R$, a sequence of convex smooth functions $(f_i)_{i\geq 1}$ can be found that is monotone in $i$ and converges to $f$ pointwise. By Dini's theorem and \cite[Proposition 2.4]{Lagerberg}, after possibly passing to a subsequence, we have a well-defined limit
\[
(\ddc f)^{\wedge d}\coloneqq\lim\limits_{i\to\infty}(\ddc f_i)^{\wedge d}
\]
weakly as currents.

Let $C$ be a tropical cycle of dimension $d\leq n$ and let $\sigma$ be a maximal open face of $C$ with weight $m_\sigma$. Then
\[
\big((\ddc f)^{\wedge d}\wedge\delta_C\big)\big|_\sigma\coloneqq m_\sigma\, (\ddc (f|_\sigma))^{\wedge d}.
\]

\begin{prop}\label{prop:TMA on maximal faces}
    Let $f:N_\R\to\R$ be a convex function and let $C$ be a tropical cycle of dimension $d\leq n$. Then as currents
    \[
    \big((\ddc f)^{\wedge d}\wedge\delta_C\big)\big|_\sigma = d!\, m_{\sigma}\, \operatorname{MA}_\R\!\left(f|_\sigma\right)
    \]
    for every maximal open face $\sigma$ of $C$ with weight $m_\sigma$.
\end{prop}
\begin{proof}
    For every convex function $f$, there exists a sequence of convex piecewise affine functions $(f_i)_{i\geq 1}$ uniformly converging to $f$ on every compact subset, see for instance \cite[Lemma 11]{PLtoConvex}. Since $K$ is algebraically closed, we have that the value group $\Gamma$ is divisible and hence dense in $\R$. Therefore, we may assume that each $f_i$ is a convex $\Gamma$-rational piecewise affine function.

    On one hand, the currents $(\ddc (f_i|_\sigma))^{\wedge d}$ weakly converge to $(\ddc (f|_\sigma))^{\wedge d}$ by \cite[Proposition 2.7]{Lagerberg}. On the other hand, by the continuity theorem of Rauch-Taylor \cite{RT77}, the real Monge-Amp{\`e}re measure $\operatorname{MA}_\R(f|_\sigma)$ is the weak limit of $\operatorname{MA}_\R(f_i|_\sigma)$. Therefore, the conclusion follows from \Cref{lem:ddf and MA} by taking limits on both sides.
\end{proof}

\begin{cor}\label{cor:trop MA on sigma}
    Let $\lVert\cdot\rVert$ be a semipositive toric metric on $L$ and let $g_{\trop}$ be its corresponding convex tropical Green function. Then
    \[
    (\trop_X)_*(c_1(L|_X,\lVert\cdot\rVert)^{\wedge d})\big|_\sigma= d!\, m_{\bar{\sigma}}\, \operatorname{MA}_\R(g_{\trop}|_\sigma)
    \]
    for every maximal open face $\sigma$ of $\trop(X^{\an}\cap T^{\an})$.
\end{cor}
\begin{proof}
    Applying \Cref{lem:approximated semipos metric by model metrics}, we obtain a sequence of toric model metrics $(\lVert\cdot\rVert_i)_{i\geq 1}$ converging to $\lVert\cdot\rVert$ such that the corresponding sequence of tropical Green functions $(g_i)_{i\geq 1}$ converging to $g_{\trop}$ by uniform convergence. By the projection formula (\ref{lem:projection formula}), we get an identity of currents
    \[
    \big( (\trop_X)_*(c_1(L|_X,\lVert\cdot\rVert_i)^{\wedge d}) \big)\big|_{N_\R}=(\ddc g_i)^{\wedge d}\wedge\delta_{\trop(X^{\an}\cap T^{\an})}.
    \]
    Restricting to the open face $\sigma$ and taking limits on both sides, we have
    \[
    \big( (\trop_X)_*(c_1(L|_X,\lVert\cdot\rVert)^{\wedge d}) \big)\big|_{\sigma}=\left((\ddc g_{\trop})^{\wedge d}\wedge\delta_{\trop(X^{\an}\cap T^{\an})}\right)\big|_\sigma.
    \]
    Then the desired result follows from \Cref{prop:TMA on maximal faces}.
\end{proof}

We have also the following identification of tropicalized Monge-Amp{\`e}re measures with polyhedral Monge-Amp{\`e}re measures.

\begin{prop}\label{prop: trop MA}
    Let $\lVert\cdot\rVert$ be a semipositive toric metric on $L$ and let $g_{\trop}$ be its corresponding convex tropical Green function. Then
    \[
    (\trop_X)_*(c_1(L|_X,\lVert\cdot\rVert)^{\wedge d})\big|_{N_\R}=\operatorname{MA}_{\operatorname{poly},\trop(X^{\an}\cap T^{\an})}(g_{\trop}).
    \]
\end{prop}
\begin{proof}
    The proof is essentially the same as that of \Cref{cor:trop MA on sigma}. In the model case, the conclusion follows from the projection formula (\ref{lem:projection formula}) and the piecewise affine definition of the polyhedral Monge-Amp{\`e}re measure. The general case follows by approximation, using the weak continuity of non-archimedean Monge-Amp{\`e}re measures and \Cref{Prop:MA poly weak convergence}.
\end{proof}

\begin{rmk}\label{rmk: MApoly}
    
    We remark that the polyhedral identity underlying \Cref{cor:trop MA on sigma} is the $\Gamma$-rational analogue of \cite[Proposition 4.20]{MApoly}. More precisely, using the notion of polyhedral Monge-Amp{\`e}re measures given in \Cref{Section: poly MA}, we have
    \[
    \operatorname{MA}_{\operatorname{poly},\sigma}(g_{\trop}|_{\sigma})=d!\, m_{\bar{\sigma}}\, \operatorname{MA}_\R(g_{\trop}|_\sigma)
    \]
    for every maximal open face $\sigma$ of $\trop(X^{\an}\cap T^{\an})$.
    
    Indeed, in the proof of \cite[Proposition 4.20]{MApoly}, we may replace the piecewise affine case used there with its $\Gamma$-rational counterpart given in \Cref{prop:ddf and MA on Rn}. The general case follows from \Cref{Prop:MA poly weak convergence}, together with the weak continuity of the classical Monge-Amp{\`e}re measure.

    Combining this identity with \Cref{prop: trop MA} and the locality of polyhedral Monge-Amp{\`e}re measures we established in \Cref{Section: poly MA} yields \Cref{cor:trop MA on sigma}.
\end{rmk}

\begin{rmk}\label{counterexample: lower dimensional face}
    In general, \Cref{cor:trop MA on sigma} can not be extended to all open faces of $\trop(X^{\an}\cap T^{\an})$. We illustrate this by constructing a counterexample of dimension $2$.

    Let $C=[(\R^2,1)]$ be the trivial tropical cycle and let
    \[
    f=\operatorname{max}\{0,x+y-1,2x+y-3,-3y,x-3y-1,2x-3y-3\}.
    \]
    Then $f$ is convex because it is the maximum of a set of convex functions.

\begin{figure}[H]
\centering
\begin{tikzpicture}[scale=0.67,transform shape,x=1.1cm,y=1.1cm,>=Latex,line cap=round,line join=round]
  \colorlet{segcol}{blue!70!black}   
  \colorlet{funccol}{red!75!black}   
  \tikzset{func/.style={text=funccol}}
  \draw[very thick,->] (-1,0) -- (9,0) node[below right] {$x$};
  \draw[very thick,->] (0,-3.5) -- (0,4.5) node[above left] {$y$};
  \node[below left] at (0,0) {$O$};

  \coordinate (P1) at (3,0);
  \coordinate (P2) at (6,0);

  \draw[very thick,segcol] (P1) -- ++(0,-3.0);
  \draw[very thick,segcol] (P2) -- ++(0,-3.0);
  \draw[very thick,segcol] (P2) -- ++(0,3.0);

  \draw[very thick,segcol] (0,3) -- (P1);

  \draw[very thick,segcol] (0,0) -- (P1);     
  \draw[very thick,segcol] (P1) -- (P2);      
  \draw[very thick,segcol] (P2) -- (8.5,0);   

  \node[text=segcol] at (1.5,0.18) {$\alpha_1$};
  \node[text=segcol] at (4.5,0.18) {$\alpha_2$};
  \node[text=segcol] at (7.2,0.18) {$\alpha_3$};

  \node[rotate=-45,text=segcol] at (1.7,1.55) {$\beta_1$};
  \node[right,text=segcol] at ($(P2)+(0,1.4)$) {$\beta_2$};
  \node[left,text=segcol]  at ($(P1)+(0,-1.6)$) {$\beta_3$};
  \node[right,text=segcol] at ($(P2)+(0,-1.6)$) {$\beta_4$};

  \node[func] at (1.2,1.0) {$0$};
  \node[func] at (3.7,1.7) {$x+y-1$};
  \node[func] at (7.4,1.8) {$2x+y-3$};

  \node[func] at (1.2,-1.4) {$-3y$};
  \node[func] at (4.3,-1.25) {$x-3y-1$};
  \node[func] at (7.4,-1.4) {$2x-3y-3$};

  \fill[green!50!black] (P1) circle (2.2pt);
  \fill[green!50!black] (P2) circle (2.2pt);
  \node[text=green!50!black,below right] at (P1) {$P_1$};
  \node[text=green!50!black,below right] at (P2) {$P_2$};
\end{tikzpicture}
\caption{The polyhedral complex given by $f$ and $C$}
\label{fig:counterexample}
\end{figure}
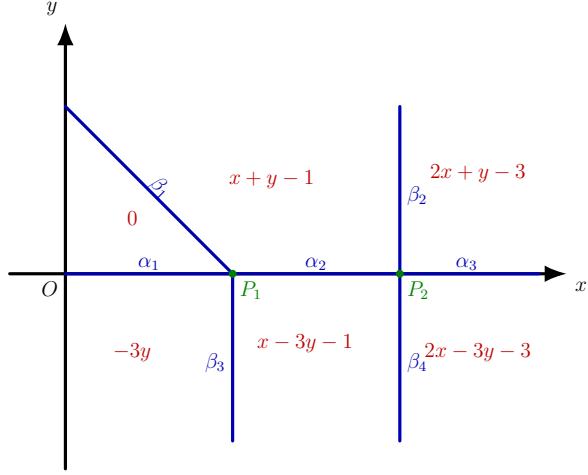

For $x>0$, the polyhedral complex given by $f$ and $C$ is shown in \Cref{fig:counterexample}. We denote the two vertices by $P_1$ and $P_2$; the one-dimensional faces on the $x$-axis by $\alpha_1,\alpha_2$ and $\alpha_3$; the other one-dimensional faces by $\beta_1,\beta_2,\beta_3$ and $\beta_4$.

We compute $\big((\ddc f)^{\wedge 2}\wedge C\big)(\{(x,y):x>0,y=0\})$ and $\operatorname{MA}(f|_{\{(x,y):x>0,y=0\}})$ using tropical intersection theory \cite[Lemma 4.9]{MihatschTIT}. 

Let $n_{\alpha_i,P_j}$ (resp. $n_{\beta_i,P_j}$) be the normal vector of $\alpha_i$ relative to $P_j$ (resp. $\beta_i$ relative to $P_j$). Denote the slopes
\[
k_{\alpha_i,P_j}\coloneqq\frac{\partial f}{\partial n_{\alpha_i,P_j}},\quad k_{\beta_i,P_j}\coloneqq\frac{\partial f}{\partial n_{\beta_i,P_j}}.
\]
Then
\begin{align*}
    k_{\alpha_1,P_1}=0,\ k_{\alpha_2,P_1}=1,\ k_{\alpha_2,P_2}=-1,\ k_{\alpha_3,P_2}=2,\\ k_{\beta_1,P_1}=0,\ k_{\beta_3,P_1}=3,\ k_{\beta_2,P_2}=1,\ k_{\beta_4,P_2}=3.
\end{align*}
Let $m(\alpha_i)$ (resp. $m(\beta_i)$) be the weight of the one-dimensional faces $\alpha_i$ (resp. $\beta_i$) in the tropical cycle $(\ddc f)\wedge C$. One can check that 
\[
m(\alpha_1)=3,\ m(\alpha_2)=4,\ m(\alpha_3)=4,\ m(\beta_1)=m(\beta_2)=m(\beta_3)=m(\beta_4)=1.
\]
It follows that
\[
\big((\ddc f)^{\wedge 2}\wedge C\big)|_{\{(x,y):x>0,y=0\}}=7\delta_{P_1}+8\delta_{P_2}
\]
and
\[
\operatorname{MA}_\R\left(f|_{\{(x,y):x>0,y=0\}}\right)=\delta_{P_1}+\delta_{P_2}.
\]
Hence, for the one-dimensional open face $\tau=\{(x,y):x>0,y=0\}$ of $C$, there does not exist a constant $m$ such that
\[
\big((\ddc f)^{\wedge 2}\wedge\delta_C\big)\big|_\tau = m\, \operatorname{MA}_\R\left(f|_\tau\right).
\]

However, the precise coefficient of $(\ddc f)^d\wedge C$ can always be computed using tropical intersection theory. In addition to \cite[Lemma 4.9]{MihatschTIT}, one may also apply the fan displacement rule \cite[Section 4.3]{MihatschTIT} or the Allermann-Rau's intersection product \cite[Section 3]{AllermannRau}.
     
\end{rmk}

We will next introduce a class of toric metrics for which the corresponding tropicalized Monge-Amp{\`e}re measure can be studied face by face. Before doing so, we notice that (\ref{lem:projection formula}) can be extended to toric metrics whose tropical Green functions are convex and piecewise smooth.

\begin{cor}\label{cor: PS projection formula}
    Let $\lVert\cdot\rVert$ be a semipositive toric metric on $L$ such that its corresponding tropical Green function $g_{\trop}:N_\R\to\R$ is piecewise smooth. Then
    \[
    (\trop_X)_*(c_1(L|_X,\lVert\cdot\rVert)^{\wedge d})\big|_{N_\R}=(\ddc g_{\trop})^{\wedge d}\wedge\delta_{\trop(X^{\an}\cap T^{\an})}.
    \]
\end{cor}
\begin{proof}
    This follows immediately from \Cref{prop: trop MA} and \Cref{prop: PS comparison of poly MA and delta form}.
\end{proof}

\begin{rmk}\label{tropical MA special class}
    We next give an explicit facewise formula of tropicalized Monge-Amp{\`e}re measures for a special class of semipositive toric metrics.

    Let us consider the kind of piecewise smooth tropical Green function $g_{\trop}$ such that $g_{\trop}=-\lambda\psi+\rho$ for $\lambda\geq 0$ and a smooth function $\rho$ on $N_\R$. After passing to a refinement of $\trop(X^{\an}\cap T^{\an})$ if necessary, we may assume that $\psi$ is affine on every face of $\trop(X^{\an}\cap T^{\an})$. Then
    \[
    (\ddc g_{\trop})^{\wedge d}\wedge\delta_{\trop(X^{\an}\cap T^{\an})}=\sum_{k=0}^d\binom{d}{k}\lambda^k(-\ddc \psi)^{\wedge k}\wedge(\ddc \rho)^{\wedge(d-k)}\wedge\delta_{\trop(X^{\an}\cap T^{\an})}.
    \]
    We denote the $(d-k)$-dimensional tropical cycle $(-\ddc \psi)^{\wedge k}\wedge\delta_{\trop(X^{\an}\cap T^{\an})}$ by $C_k$. Then by the smoothness of $\rho$, for every maximal open face $\tau$ of $C_k$ we have
    \[
    \big((\ddc \rho)^{\wedge(d-k)}\wedge\delta_{C_k}\big)\big|_\tau=(d-k)!\, m_\tau\, \operatorname{MA}_\R(\rho|_\tau),
    \]
    where $m_\tau$ is the weight of $\tau$ in $C_k$ and $\operatorname{MA}(\rho|_\tau)$ is defined to be the full-dimensional differential form whose coefficient is the determinant of the Hessian matrix of $\rho|_\tau$ \cite[Lemma 1.7.1]{2Antoine}. If we regard the tropical cycle $\trop(X^{\an}\cap T^{an})$ as the union of all the faces of $C_k$ for all $k$, then
    \[
    \big((\ddc g_{\trop})^{\wedge d}\wedge\delta_{\trop(X^{\an}\cap T^{\an})}\big)\big|_\tau=\frac{d!}{k!}\, \lambda^k\, m_\tau\, \operatorname{MA}_\R(\rho|_\tau)
    \]
    for every open face $\tau$ of $\trop(X^{\an}\cap T^{\an})$ of dimension $d-k$. Applying \Cref{cor: PS projection formula}, we get
    \[
    \trop_*(c_1(L|_X,\lVert\cdot\rVert)^{\wedge d})\big|_\tau=\frac{d!}{k!}\, \lambda^k\, m_\tau\, \operatorname{MA}_\R(\rho|_\tau).
    \]
    When $\lambda=0$, we have $g_{\trop}=\rho$ is a smooth function, and only the term $k=0$ occurs in the above formula. In particular, the corresponding measure has no contribution on lower-dimensional faces.
\end{rmk}

\section{Monge-Amp{\`e}re measures on tropical skeletons}\label{Section: main thm}

We keep the notation and assumptions of \Cref{Section: NA MA measure as delta current} and \Cref{Section: tropicalized MA}. Thus, $X$ is a $d$-dimensional closed subvariety of a proper toric variety $Y=Y_\Sigma$ with dense torus $T$ over $K$. We denote the closed embedding by $i:X\hookrightarrow Y$ and assume that $X\cap T\neq\varnothing$. Since only the restricted tropicalization map $(X\cap T)^{\an}\xhookrightarrow{i^{\an}}T^{\an}\to N_\R$ will be used in this section, we denote it simply by $\trop$ throughout.

Let $L$ be a toric line bundle on $Y$ and let $\lVert\cdot\rVert$ be a semipositive toric metric on $L$ with associated convex tropical Green function $g_{\trop}:N_\R\to\R$. 

In \Cref{Section: tropicalized MA}, we give a description of the tropicalized Monge-Amp{\`e}re measure in terms of the real Monge-Amp{\`e}re measure. In this section, we are going to describe the Monge-Amp{\`e}re measure $c_1(L|_X,\lVert\cdot\rVert)^{\wedge d}$ locally as the tropical pullback of a polyhedral Monge-Amp{\`e}re measure discussed in \Cref{Section: poly MA} of the tropicalization of a local germ. On each local maximal open face, the latter is a weighted multiple of the real Monge-Amp{\`e}re measure. By the projection formula, the tropical pullback description for $c_1(L|_X,\lVert\cdot\rVert)^{\wedge d}$ leads to a description of the tropicalized Monge-Amp{\`e}re measure in terms of the polyhedral Monge-Amp{\`e}re measure.

As we have seen in \Cref{Section: NA MA measure as delta current}, the boundary $X\backslash T$ has measure zero with respect to $c_1(L|_X,\lVert\cdot\rVert)^{\wedge d}$, and the restriction $c_1(L|_X,\lVert\cdot\rVert)^{\wedge d}|_{X^{\an}\cap T^{\an}}$ is supported on the tropical skeleton $\Sigma(X^{\an}\cap T^{\an},i^{\an}|_{X^{\an}\cap T^{\an}})$.

For simplicity, in this section, we will write
\[
\Sigma(U)\coloneqq\Sigma(U,i^{\an}|_U)
\]
for every analytic space $U$ contained in $X^{\an}\cap T^{\an}$.

The counterexample in \Cref{counterexample: lower dimensional face} motivates us to study $c_1(L|_X,\lVert\cdot\rVert)^{\wedge d}$ in some local polyhedral charts of $\Sigma(X^{\an}\cap T^{\an})$ at a specific point. The definition of local polyhedral charts was given in \Cref{subsection:local charts of tropical skeletons}. There is another reason for us to work with local polyhedral charts. The tropicalization map $\trop$ need not be injective on the tropical skeleton, and distinct branches may map to the same subset of $N_\R$. In particular, different points of the tropical skeleton may have the same tropicalization, while carrying different local geometric data. The following example illustrates this phenomenon. 

\begin{ex}\label{example: different branches}
    This example is a revisit of \cite[Example 2.6]{BPRpreprint} from a local point of view. The skeleton considered there coincides with the tropical skeleton in our sense, since the tropicalization map below is nonconstant on every edge; see \cite[Remark 3.14]{TropicalLinearIso}. Let $p$ be a prime and let $K=\C_p$, with the valuation normalized by $\operatorname{val}(p)=1$. Consider the curve $X\subset\mathbb{G}_m^2$ given by the parametrization
    \[
    x(t)=t(t-p)\quad\text{and}\quad y(t)=t-1.
    \]
    The tropicalization map from the tropical skeleton of $X^{\an}$ to $N_\R$ is illustrated by \Cref{fig:example trop from skeletons}.

    \begin{figure}[!htbp]
        \centering
\begin{tikzpicture}[
  scale=0.67,
  transform shape,
  x=1.1cm,
  y=1.1cm,
  >=Latex,
  line cap=round,
  line join=round
]
  \colorlet{eone}{magenta!75!black}
  \colorlet{etwo}{blue!70!black}
  \colorlet{ethree}{green!50!black}
  \colorlet{efour}{orange!85!black}
  \colorlet{efive}{brown!75!black}

  \coordinate (z)  at (0,0);
  \coordinate (zp) at (3,0);

  \draw[very thick,efour]
    (z) -- (0,2.6);

  \draw[very thick,efive]
    (z) -- (-2.2,-1.5);

  \draw[very thick,etwo]
    (z) -- (zp);

  \draw[very thick,ethree]
    (zp) -- (4.45,2.45);

  \draw[very thick,eone]
    (zp) -- (4.55,-2.05);

  \fill (z)  circle (2pt);
  \fill (zp) circle (2pt);

  \node[below right=2pt] at (z)
    {$\zeta$};

  \node[below left] at (zp)
    {$\zeta'$};

  \node[above] at (0,2.6)
    {$1$};

  \node[above right] at (4.45,2.45)
    {$0$};

  \node[below right] at (4.55,-2.05)
    {$p$};

  \node[below left] at (-2.2,-1.5)
    {$\infty$};

  \node[text=efour,left] at (0,1.35)
    {$e_4$};

  \node[text=efive,above left] at (-1.1,-0.75)
    {$e_5$};

  \node[text=etwo,above] at (1.5,0)
    {$e_2$};

  \node[text=ethree,above left] at (3.75,1.2)
    {$e_3$};

  \node[text=eone,above right] at (3.8,-1.05)
    {$e_1$};

  \draw[very thick,->]
    (5.25,0) -- (6.45,0)
    node[midway,above=2pt] {$\trop$};

  \begin{scope}[shift={(9.0,0)}]
    \draw[very thick]
      (-2.2,-1.5) -- (0,0) -- (0,2.8);

    \draw[very thick]
      (0,0) -- (2,0);

    \draw[very thick,efour]
      (0,0) -- (0,2.8);

    \draw[very thick,efive]
      (0,0) -- (-2.2,-1.5);

    \draw[very thick,etwo]
      (0,0) -- (2,0);

    \begin{scope}
      \clip (1.9,0) rectangle (5.1,0.15);
      \draw[line width=2.8pt,ethree]
        (2,0) -- (5,0);
    \end{scope}

    \begin{scope}
      \clip (1.9,-0.15) rectangle (5.1,0);
      \draw[line width=2.8pt,eone]
        (2,0) -- (5,0);
    \end{scope}

    \fill (0,0) circle (2pt);
    \fill (2,0) circle (2pt);

    \node[below right] at (0,0)
      {$(0,0)$};

    \node[below] at (2,0)
      {$(2,0)$};

    \node[right] at (0,1)
      {$(0,1)$};

    \node[below=9pt] at (-1.45,-1.0)
      {$(-2,-1)$};
  \end{scope}
\end{tikzpicture}
        \caption{The tropicalization of the curve $X$}
        \label{fig:example trop from skeletons}
    \end{figure}
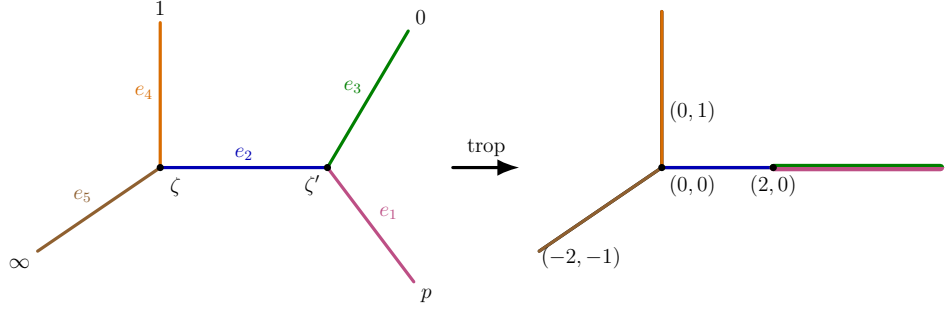

    By the parameter $t$, we can identify $X$ with $\mathbb{P}^1_K\backslash\{0,1,p,\infty\}$. Recall that, for $a\in K$ and $r>0$, we denote by $\zeta_{a,r}$ the weighted Gauss point corresponding to the multiplicative seminorm
    \[
    \left| \sum_ic_i(t-a)^i \right|_{\zeta_{a,r}}=\operatorname{max}_i |c_i|r^i.
    \]
    In \Cref{fig:example trop from skeletons}, we have
    \begin{align*}
        &\zeta'=\zeta_{0,|p|},\\
        &e_1=\{\zeta_{p,|p|^\alpha}:\alpha\geq 1\},\\
        &e_2=\{\zeta_{0,|p|^\alpha}:0\leq\alpha\leq 1\},\\
        &e_3=\{\zeta_{0,|p|^\alpha}:\alpha\geq 1\}.
    \end{align*}
    Therefore,
    \[
    \operatorname{val}_{\zeta'}(t-b)=\operatorname{min}\{1,\operatorname{val}(-b)\}
    \]
    and thus
    \[
    \trop(\zeta')=\big( \operatorname{val}_{\zeta'}(t)+\operatorname{val}_{\zeta'}(t-p),\ \operatorname{val}_{\zeta'}(t-1) \big)=(2,0).
    \]
    Similarly, we get
    \[
    \trop(e_1)=\trop(e_3)=\{(\alpha+1,0):\alpha\geq 1\}
    \]
    and
    \[
    \trop(e_2)=\{(2\alpha,0):0\leq\alpha\leq 1\}.
    \]
    In particular, the tropical weight of $\{(\alpha+1,0):\alpha\geq 1\}$ is $1+1=2$, and the tropical weight of $\{(2\alpha,0):0\leq\alpha\leq 1\}$ is $2$.

    We now fix a point $x$ on the skeleton near $\zeta'$ and choose a sufficiently small neighborhood $U(x)$ of $x$. Depending on the open face containing $x$, there are four possible cases; see \Cref{fig:example e3}, \Cref{fig:example e2}, \Cref{fig:example e1} and \Cref{fig:example zeta}.

    We conclude the example by noting that the local data on the tropical skeleton are not determined by the tropicalization alone. The tropical skeleton has weight $1$ in $e_1$ and $e_3$ and weight $2$ in $e_2$. At $\zeta'$, three branches meet.

    \begin{figure}[H]
        \centering
        \begin{tikzpicture}[
  scale=0.67,
  transform shape,
  x=1.1cm,
  y=1.1cm,
  >=Latex,
  line cap=round,
  line join=round
]
  \colorlet{funccol}{red!75!black}
  \colorlet{eone}{magenta!75!black}
  \colorlet{etwo}{blue!70!black}
  \colorlet{ethree}{green!50!black}
  \colorlet{efour}{orange!85!black}
  \colorlet{efive}{brown!75!black}

  \coordinate (z)  at (0,0);
  \coordinate (zp) at (3,0);

  \draw[very thick,efour]
    (z) -- (0,2.6);

  \draw[very thick,efive]
    (z) -- (-2.2,-1.5);

  \draw[very thick,etwo]
    (z) -- (zp);

  \draw[very thick,ethree]
    (zp) -- (4.45,2.45);

  \draw[very thick,eone]
    (zp) -- (4.55,-2.05);

  \fill (z)  circle (2pt);
  \fill (zp) circle (2pt);

  \node[below right=2pt] at (z)
    {$\zeta$};

  \node[below left] at (zp)
    {$\zeta'$};

  \node[above] at (0,2.6)
    {$1$};

  \node[above right] at (4.45,2.45)
    {$0$};

  \node[below right] at (4.55,-2.05)
    {$p$};

  \node[below left] at (-2.2,-1.5)
    {$\infty$};

  \node[text=efour,left] at (0,1.35)
    {$e_4$};

  \node[text=efive,above left] at (-1.1,-0.75)
    {$e_5$};

  \node[text=etwo,above] at (1.5,0)
    {$e_2$};

  \node[text=ethree,above left] at (3.25,0.35)
    {$e_3$};

  \node[text=eone,above right] at (3.8,-1.05)
    {$e_1$};

  \coordinate (x) at (3.78,1.32);

  \fill[funccol]
    (x) circle (2.2pt);

  \draw[very thick,funccol]
    (x) circle (0.72);

  \node[text=funccol,right] at (4.35,1.6)
    {$U(x)$};

  \node[text=funccol,below right] at (x)
    {$x$};

  \draw[very thick,->]
    (5.25,0) -- (6.45,0)
    node[midway,above=2pt] {$\trop$};

  \begin{scope}[shift={(9.0,0)}]
    \draw[very thick]
      (-2.2,-1.5) -- (0,0) -- (0,2.8);

    \draw[very thick]
      (0,0) -- (2,0);

    \draw[very thick,ethree]
      (2,0) -- (5,0);

    \fill (0,0) circle (2pt);
    \fill (2,0) circle (2pt);

    \node[below right] at (0,0)
      {$(0,0)$};

    \node[below] at (2,0)
      {$(2,0)$};

    \node[right] at (0,1)
      {$(0,1)$};

    \node[below=9pt] at (-1.45,-1.0)
      {$(-2,-1)$};

    \coordinate (tx) at (4.0,0);

    \fill[funccol]
      (tx) circle (2.2pt);

    \draw[very thick,funccol]
      (tx) circle (0.72);

    \node[text=funccol,above] at (4.0,0.85)
      {$\trop(U(x))$};

    \node[text=funccol,below] at (4.0,-0.18)
      {$\trop(x)$};
  \end{scope}
\end{tikzpicture}
        \caption{The case where $x\in e_3^\circ$}
        \label{fig:example e3}
    \end{figure}
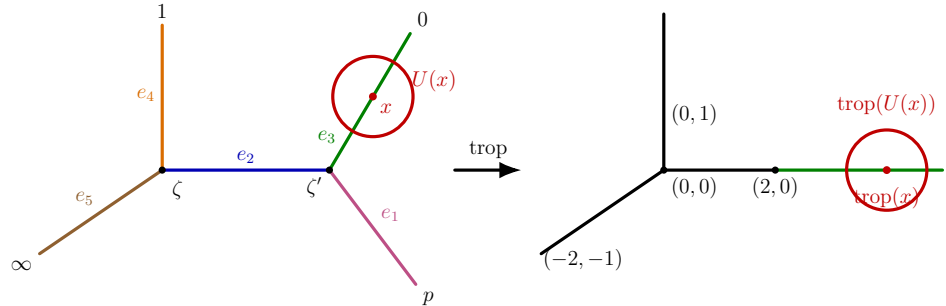

    \begin{figure}[H]
        \centering
        \begin{tikzpicture}[
  scale=0.67,
  transform shape,
  x=1.1cm,
  y=1.1cm,
  >=Latex,
  line cap=round,
  line join=round
]
  \colorlet{funccol}{red!75!black}
  \colorlet{eone}{magenta!75!black}
  \colorlet{etwo}{blue!70!black}
  \colorlet{ethree}{green!50!black}
  \colorlet{efour}{orange!85!black}
  \colorlet{efive}{brown!75!black}

  \coordinate (z)  at (0,0);
  \coordinate (zp) at (3,0);

  \draw[very thick,efour]
    (z) -- (0,2.6);

  \draw[very thick,efive]
    (z) -- (-2.2,-1.5);

  \draw[very thick,etwo]
    (z) -- (zp);

  \draw[very thick,ethree]
    (zp) -- (4.45,2.45);

  \draw[very thick,eone]
    (zp) -- (4.55,-2.05);

  \fill (z)  circle (2pt);
  \fill (zp) circle (2pt);

  \node[below right=2pt] at (z)
    {$\zeta$};

  \node[below left] at (zp)
    {$\zeta'$};

  \node[above] at (0,2.6)
    {$1$};

  \node[above right] at (4.45,2.45)
    {$0$};

  \node[below right] at (4.55,-2.05)
    {$p$};

  \node[below left] at (-2.2,-1.5)
    {$\infty$};

  \node[text=efour,left] at (0,1.35)
    {$e_4$};

  \node[text=efive,above left] at (-1.1,-0.75)
    {$e_5$};

  \node[text=etwo,above] at (0.70,0)
    {$e_2$};

  \node[text=ethree,above left] at (3.75,1.2)
    {$e_3$};

  \node[text=eone,above right] at (3.8,-1.05)
    {$e_1$};

  \coordinate (x) at (1.5,0);

  \fill[funccol]
    (x) circle (2.2pt);

  \draw[very thick,funccol]
    (x) circle (0.62);

  \node[text=funccol,above] at (1.5,0.72)
    {$U(x)$};

  \node[text=funccol,below] at (1.5,-0.18)
    {$x$};

  \draw[very thick,->]
    (5.25,0) -- (6.45,0)
    node[midway,above=2pt] {$\trop$};

  \begin{scope}[shift={(9.0,0)}]
    \draw[very thick]
      (-2.2,-1.5) -- (0,0) -- (0,2.8);

    \draw[very thick]
      (0,0) -- (5,0);

    \draw[very thick,etwo]
      (0,0) -- (2,0);

    \fill (0,0) circle (2pt);
    \fill (2,0) circle (2pt);

    \node[below right] at (0,0)
      {$(0,0)$};

    \node[below] at (2,0)
      {$(2,0)$};

    \node[left] at (0,1)
      {$(0,1)$};

    \node[below=9pt] at (-1.45,-1.0)
      {$(-2,-1)$};

    \coordinate (tx) at (1.25,0);

    \fill[funccol]
      (tx) circle (2.2pt);

    \draw[very thick,funccol]
      (tx) circle (0.45);

    \node[text=funccol,above] at (1.25,0.56)
      {$\trop(U(x))$};

    \node[text=funccol,below=11pt] at (tx)
      {$\trop(x)$};
  \end{scope}
\end{tikzpicture}
        \caption{The case where $x\in e_2^\circ$}
        \label{fig:example e2}
    \end{figure}
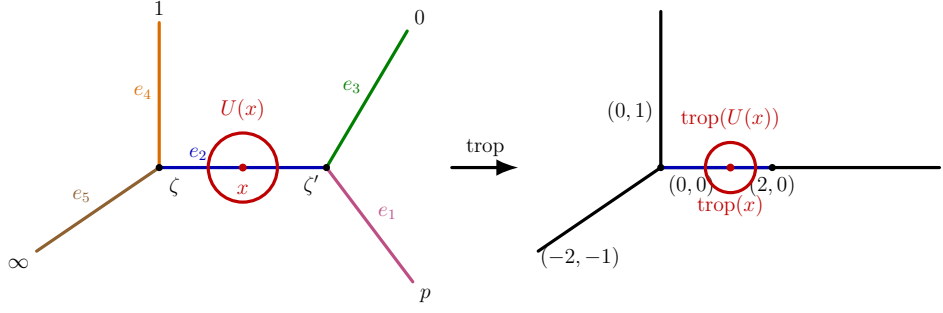

    \begin{figure}[H]
        \centering
        \begin{tikzpicture}[
  scale=0.67,
  transform shape,
  x=1.1cm,
  y=1.1cm,
  >=Latex,
  line cap=round,
  line join=round
]
  \colorlet{funccol}{red!75!black}
  \colorlet{eone}{magenta!75!black}
  \colorlet{etwo}{blue!70!black}
  \colorlet{ethree}{green!50!black}
  \colorlet{efour}{orange!85!black}
  \colorlet{efive}{brown!75!black}

  \coordinate (z)  at (0,0);
  \coordinate (zp) at (3,0);

  \draw[very thick,efour]
    (z) -- (0,2.6);

  \draw[very thick,efive]
    (z) -- (-2.2,-1.5);

  \draw[very thick,etwo]
    (z) -- (zp);

  \draw[very thick,ethree]
    (zp) -- (4.45,2.45);

  \draw[very thick,eone]
    (zp) -- (4.55,-2.05);

  \fill (z)  circle (2pt);
  \fill (zp) circle (2pt);

  \node[below right=2pt] at (z)
    {$\zeta$};

  \node[below left] at (zp)
    {$\zeta'$};

  \node[above] at (0,2.6)
    {$1$};

  \node[above right] at (4.45,2.45)
    {$0$};

  \node[below right] at (4.55,-2.05)
    {$p$};

  \node[below left] at (-2.2,-1.5)
    {$\infty$};

  \node[text=efour,left] at (0,1.35)
    {$e_4$};

  \node[text=efive,above left] at (-1.1,-0.75)
    {$e_5$};

  \node[text=etwo,above] at (1.5,0)
    {$e_2$};

  \node[text=ethree,above left] at (3.75,1.2)
    {$e_3$};

  \node[text=eone,above right] at (3.25,-0.25)
    {$e_1$};

  \coordinate (x) at (3.85,-1.12);

  \fill[funccol]
    (x) circle (2.2pt);

  \draw[very thick,funccol]
    (x) circle (0.72);

  \node[text=funccol,right] at (4.35,-0.75)
    {$U(x)$};

  \node[text=funccol,below left] at (x)
    {$x$};

  \draw[very thick,->]
    (5.25,0) -- (6.45,0)
    node[midway,above=2pt] {$\trop$};

  \begin{scope}[shift={(9.0,0)}]
    \draw[very thick]
      (-2.2,-1.5) -- (0,0) -- (0,2.8);

    \draw[very thick]
      (0,0) -- (5,0);

    \draw[very thick,eone]
      (2,0) -- (5,0);

    \fill (0,0) circle (2pt);
    \fill (2,0) circle (2pt);

    \node[below right] at (0,0)
      {$(0,0)$};

    \node[below] at (2,0)
      {$(2,0)$};

    \node[right] at (0,1)
      {$(0,1)$};

    \node[below=9pt] at (-1.45,-1.0)
      {$(-2,-1)$};

    \coordinate (tx) at (4.0,0);

    \fill[funccol]
      (tx) circle (2.2pt);

    \draw[very thick,funccol]
      (tx) circle (0.72);

    \node[text=funccol,above] at (4.0,0.85)
      {$\trop(U(x))$};

    \node[text=funccol,below] at (4.0,-0.18)
      {$\trop(x)$};
  \end{scope}
\end{tikzpicture}
        \caption{The case where $x\in e_1^\circ$}
        \label{fig:example e1}
    \end{figure}
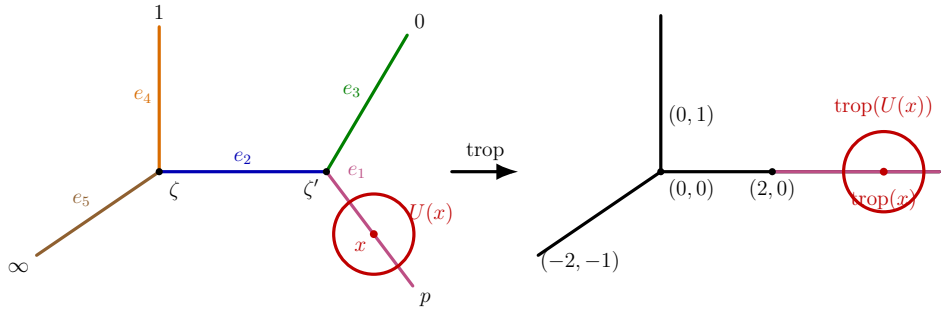

    \begin{figure}[H]
        \centering
        \begin{tikzpicture}[
  scale=0.67,
  transform shape,
  x=1.1cm,
  y=1.1cm,
  >=Latex,
  line cap=round,
  line join=round
]
  \colorlet{funccol}{red!75!black}
  \colorlet{eone}{magenta!75!black}
  \colorlet{etwo}{blue!70!black}
  \colorlet{ethree}{green!50!black}
  \colorlet{efour}{orange!85!black}
  \colorlet{efive}{brown!75!black}

  \coordinate (z)  at (0,0);
  \coordinate (zp) at (3,0);

  \draw[very thick,efour]
    (z) -- (0,2.6);

  \draw[very thick,efive]
    (z) -- (-2.2,-1.5);

  \draw[very thick,etwo]
    (z) -- (zp);

  \draw[very thick,ethree]
    (zp) -- (4.45,2.45);

  \draw[very thick,eone]
    (zp) -- (4.55,-2.05);

  \fill (z)  circle (2pt);
  \fill (zp) circle (2pt);

  \node[below right=2pt] at (z)
    {$\zeta$};

  \node[below left] at (zp)
    {$\zeta'$};

  \node[above] at (0,2.6)
    {$1$};

  \node[above right] at (4.45,2.45)
    {$0$};

  \node[below right] at (4.55,-2.05)
    {$p$};

  \node[below left] at (-2.2,-1.5)
    {$\infty$};

  \node[text=efour,left] at (0,1.35)
    {$e_4$};

  \node[text=efive,above left] at (-1.1,-0.75)
    {$e_5$};

  \node[text=etwo,above] at (1.5,0)
    {$e_2$};

  \node[text=ethree,above left] at (3.75,1.2)
    {$e_3$};

  \node[text=eone,above right] at (3.8,-1.05)
    {$e_1$};

  \draw[very thick,funccol]
    (zp) circle (0.85);

  \node[text=funccol,right] at (3.7,0.65)
    {$U(\zeta')$};

  \draw[very thick,->]
    (5.25,0) -- (6.45,0)
    node[midway,above=2pt] {$\trop$};

  \begin{scope}[shift={(9.0,0)}]
    \draw[very thick]
      (-2.2,-1.5) -- (0,0) -- (0,2.8);

    \draw[very thick]
      (0,0) -- (2,0);

    \draw[very thick,etwo]
      (0,0) -- (2,0);

    \begin{scope}
      \clip (1.9,0) rectangle (5.1,0.15);
      \draw[line width=2.8pt,ethree]
        (2,0) -- (5,0);
    \end{scope}

    \begin{scope}
      \clip (1.9,-0.15) rectangle (5.1,0);
      \draw[line width=2.8pt,eone]
        (2,0) -- (5,0);
    \end{scope}

    \fill (0,0) circle (2pt);
    \fill (2,0) circle (2pt);

    \node[below right] at (0,0)
      {$(0,0)$};

    \node[below] at (2,0)
      {$(2,0)$};

    \node[right] at (0,1)
      {$(0,1)$};

    \node[below=9pt] at (-1.45,-1.0)
      {$(-2,-1)$};

    \draw[very thick,funccol]
      (2,0) circle (1.05);

    \node[text=funccol,above] at (2.65,1.0)
      {$\trop(U(\zeta'))$};
  \end{scope}
\end{tikzpicture}
        \caption{The case where $x=\zeta'$}
        \label{fig:example zeta}
    \end{figure}
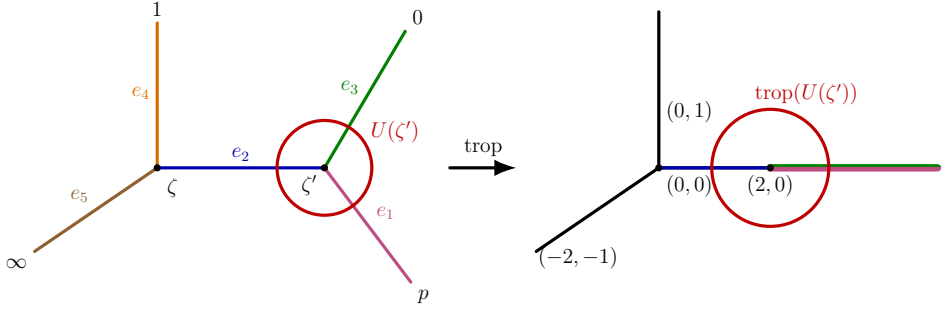
\end{ex}

Indeed, if $\lVert\cdot\rVert$ is a model metric, the facewise formula for $c_1(L|_X,\lVert\cdot\rVert)^{\wedge d}$ can be given by tropical intersection product after restricting the measure to a specified neighborhood of each point. At every point $x\in\Sigma(X^{\an}\cap T^{\an})$, it was shown in \cite[Lemma 8.21]{TropicalLinearIso} based on results of Ducros \cite[Theorem 3.4]{Duc12} that there is a compact strictly analytic neighborhood $V(x)$ of $x$ such that for any compact strictly analytic neighborhood $W\subset V(x)$ of $x$, the weighted germs of $\trop_*(\Sigma(V(x)))$ and $\trop_*(\Sigma(W))$ at $\trop(x)$ agree. Following \cite[8.22]{TropicalLinearIso}, we define the tropical variety $\trop(X^{\an}\cap T^{\an},x)$ of the germ $(X^{\an}\cap T^{\an},x)$ as the germ of $\trop_*(\Sigma(V(x)))$ at $\trop(x)$. It is a tropical fan of pure dimension $d$ with nonnegative weights on its $d$-faces. By \cite[Theorem 9.7(1)]{TropicalLinearIso}, the weighted fan $\trop(X^{\an}\cap T^{\an},x)$ is balanced.

We remark that
\[
|\trop_*(\Sigma(V(x)))|=\trop(V(x))\quad\text{and}\quad|\trop_*(\Sigma(W))|=\trop(W)
\]
by (\ref{eq:trop on skeleton which is surj}).

Recall that, as stated in \Cref{subsection:local charts of tropical skeletons}, for every $x\in\Sigma(X^{\an}\cap T^{\an})$, there exists a local $(\Z,\Gamma)$-polyhedral chart $\mathcal{P}_x$ of $\Sigma(X^{\an}\cap T^{\an})$ at $x$ compatible with $\trop|_{X^{\an}\cap T^{\an}}$. We fix one such chart $\mathcal{P}_x$. More precisely, $\mathcal{P}_x$ is given by $\Sigma(V_i)$, where $V_i=\trop^{-1}(\Delta_i)$ for some polytopes $\Delta_i\subset N_\R$. The arguments below apply equally to any other compatible local polyhedral chart at $x$.

In the proof of \Cref{thm:main thm}, the integral of a compactly supported polyhedral current $T$ on a piecewise linear space $\mathcal{P}$ is defined as
\[
\int_{\mathcal{P}} T\coloneqq T(1).
\]

We now prove \Cref{thm:intro main}.

\begin{thm}\label{thm:main thm}
    Let $x\in\Sigma(X^{\an}\cap T^{\an})$. Then there exists an open neighborhood $U(x)$ of $x$ such that the following holds.

    For every toric line bundle $L$ and every semipositive toric metric $\lVert\cdot\rVert$ on $L$ with tropical Green function $g_{\trop}:N_\R\to\R$, we have
    \begin{equation}\label{eq:local description of MA on U}
        \left(\trop|_{U(x)}\right)_*\left(c_1(L|_X,\lVert\cdot\rVert)^{\wedge d}|_{U(x)}\right)=\operatorname{MA}_{\operatorname{poly},C_x}\left(g_{\trop}\right),
    \end{equation}
    as measures, where $C_x$ is a fixed weighted representative of $\trop(X^{\an}\cap T^{\an},x)$ with support $\trop(U(x))$.

    Moreover, for the open face $\Delta$ of $\mathcal{P}_x$ such that $x\in\Delta$,
    \begin{equation}\label{eq:local description of MA on faces and germs}
        c_1(L|_X,\lVert\cdot\rVert)^{\wedge d}|_{\Delta\cap U(x)}=(\trop|_{\Delta\cap U(x)})^*(\operatorname{MA}_{\operatorname{poly},C_x}(g_{\trop})|_{\trop(\Delta\cap U(x))}).
        \end{equation}
\end{thm}
\begin{proof}
    Let $\Delta$ be the open face of $\mathcal{P}_x$ such that $x\in\Delta$. If the construction of $U(x)$ below gives us
    \begin{equation}\label{eq:fiber of Delta cap Ux}
        (\trop|_{\Sigma(U(x))})^{-1}(\trop(\Delta\cap U(x)))=\Delta\cap U(x),
    \end{equation}
    then the identity (\ref{eq:local description of MA on faces and germs}) follows immediately from (\ref{eq:local description of MA on U}) by evaluating both sides of the Radon measures on Borel subsets of $\Delta\cap U(x)$. 
    
    First, for every $x\in \Sigma(X^{\an}\cap T^{\an})$, we are going to construct the open neighborhood $U(x)$ of $x$.
    
    Let $V(x)$ be the compact strictly analytic neighborhood that defines the germ $\trop(X^{\an}\cap T^{\an},x)$ as discussed above, following \cite[Lemma 8.21 and 8.22]{TropicalLinearIso}. By the construction of $V(x)$ in \cite[Lemma 8.21]{TropicalLinearIso}, we may assume that $\Sigma(V(x))\subset\mathcal{P}_x$ is an abstract $(\Z,\Gamma)$-polyhedral complex and $V(x)$ does not intersect with any faces of $\mathcal{P}_x$ not containing $x$. As stated in \cite[10.5]{TropicalLinearIso}, open subsets whose closures are compact strictly analytic domains form a basis of the analytic topology of $X^{\an}\cap T^{\an}$. Thus, we can take a compact strictly analytic neighborhood $W(x)$ of $x$ such that $W(x)\subset V(x)^\circ$. Since the germs of $\trop_*(\Sigma(W(x)))$ and $\trop_*(\Sigma(V(x)))$ at $\trop(x)$ are both $\trop(X^{\an}\cap T^{\an},x)$, there exists an open neighborhood $\Omega_1$ of $\trop(x)$ such that
    \[
    \trop(W(x))\cap\Omega_1=\trop(V(x))\cap\Omega_1.
    \]
    On the other hand, by the construction of $V(x)$, we have
    \[
    \trop^{-1}(\trop(x))\cap\Sigma(V(x))=\{x\}.
    \]
    Indeed, suppose $y\in\Sigma(V(x))$ and $\trop(y)=\trop(x)$. Let $F$ be the minimal closed face of $\mathcal{P}_x$ containing $y$. Since $\Sigma(V(x))$ is disjoint from every face of $\mathcal{P}_x$ not containing $x$, we have that $x\in F$. By the injectivity of $\trop|_{F}$ given in \Cref{Section: tropical skeletons}, we have $y=x$, which is the desired identity. Therefore, it follows that
    \[
    \trop(x)\notin\trop(\Sigma(V(x))\backslash V(x)^{\circ}).
    \]
    Since $\Sigma(V(x))\backslash V(x)^{\circ}$ is compact and thus $\trop(\Sigma(V(x))\backslash V(x)^{\circ})$ is closed, there exists an open neighborhood $\Omega_2$ of $\trop(x)$ such that
    \begin{equation}\label{eq:construction of U p1}
        \Omega_2\cap\trop(\Sigma(V(x))\backslash V(x)^{\circ})=\varnothing.
    \end{equation}
    Set $\Omega\coloneqq\Omega_1\cap\Omega_2$ and let
    \[
    U(x)\coloneqq V(x)^{\circ}\cap\trop^{-1}(\Omega).
    \]
    We claim that
    \begin{equation}\label{eq:construction of U p2}
        \trop(U(x))=\trop(V(x))\cap\Omega
    \end{equation}
    and
    \begin{equation}\label{eq:construction of U p3}
        \Sigma(U(x))=\Sigma(V(x))\cap\trop^{-1}(\Omega).
    \end{equation}
    
    Indeed, it is clear that $\trop(U(x))\subset\trop(V(x))\cap\Omega$. Conversely, suppose $\omega\in\trop(V(x))\cap\Omega=\trop(W(x))\cap\Omega$. Then there exists $y\in W(x)$ such that $\omega=\trop(y)$. Since $y\in W(x)\cap\trop^{-1}(\Omega)\subset V(x)^\circ\cap\trop^{-1}(\Omega)=U(x)$, we have the identity (\ref{eq:construction of U p2}).

    To prove (\ref{eq:construction of U p3}), we notice that
    \[
    \Sigma(U(x))=\Sigma(V(x))\cap U(x)=\Sigma(V(x))\cap V(x)^\circ\cap\trop^{-1}(\Omega).
    \]
    It follows from (\ref{eq:construction of U p1}) that
    \[
    \Sigma(V(x))\cap\trop^{-1}(\Omega)\subset V(x)^\circ
    \]
    and thus (\ref{eq:construction of U p3}) holds. The claim is proved.

    Set
    \[
    C_x\coloneqq\trop_*(\Sigma(V(x)))|_\Omega=\trop_*(\Sigma(U(x))).
    \]
    Then $C_x$ is a weighted open representative of $\trop(X^{\an}\cap T^{\an},x)$ with support $\trop(U(x))$.
    By \cite[Theorem 9.7]{TropicalLinearIso}, we have $\trop_*(\Sigma(V(x)))|_\Omega$ is the restriction of a balanced weighted $(\Z,\Gamma)$-polyhedral complex in $N_\R$ to $\Omega$. We endow $C_x$ with this polyhedral structure.
    
    We next verify (\ref{eq:fiber of Delta cap Ux}). Take $y\in(\trop|_{\Sigma(U(x))})^{-1}(\trop(\Delta\cap U(x)))$. Then $y\in\Sigma(U(x))$ and there exists a $z\in\Delta\cap U(x)$ such that $\trop(y)=\trop(z)$. Let $F$ be the minimal closed face of $\mathcal{P}_x$ containing $y$. By the choice of $V(x)$, we have $x\in F$. Since $x\in\operatorname{relint}(\overline{\Delta})$, the closed face $\overline{\Delta}$ is a face of $F$. As $\trop|_F$ is a linear isomorphism as mentioned in \Cref{Section: tropical skeletons}, we have $y=z\in\Delta\cap U(x)$, which proves (\ref{eq:fiber of Delta cap Ux}). Therefore, if (\ref{eq:local description of MA on U}) holds, then (\ref{eq:local description of MA on faces and germs}) follows.

    We are now going to prove (\ref{eq:local description of MA on U}).

    We may assume that $\lVert\cdot\rVert$ is a toric model metric and thus $g_{\trop}$ is $\Gamma$-rational piecewise affine. Indeed, by \Cref{lem:approximated semipos metric by model metrics}, for any semipositive toric metric $\lVert\cdot\rVert$ with convex tropical Green function $g_{\trop}$, there exists a sequence of toric model metrics $(\lVert\cdot\rVert_i)_{i\geq 1}$ with corresponding convex $\Gamma$-rational piecewise affine tropical Green functions $(g_i)_{i\geq 1}$ such that $(g_i)_{i\geq 1}$ converges uniformly to $g_{\trop}$ on $N_\R$ and $c_1(L|_X,\lVert\cdot\rVert)^{\wedge d}$ is the weak limit of $(c_1(L|_X,\lVert\cdot\rVert_i)^{\wedge d})_{i\geq 1}$. This implies that
    \[
    c_1(L|_X,\lVert\cdot\rVert_i)^{\wedge d}|_{U(x)}\longrightarrow c_1(L|_X,\lVert\cdot\rVert)^{\wedge d}|_{U(x)}
    \]
    weakly against compactly supported continuous functions on $U(x)$. We have shown in \Cref{prop:MA support} that the non-archimedean Monge-Amp{\`e}re measures are supported on tropical skeletons. Therefore, if the map
    \[
    \trop|_{\Sigma(U(x))}:\Sigma(U(x))\to\trop(U(x))
    \]
    is proper, then
    \begin{equation}\label{eq:pf of main thm the properness of restriction}
        \left(\trop|_{U(x)}\right)_*\left(c_1(L|_X,\lVert\cdot\rVert_i)^{\wedge d}|_{U(x)}\right)\longrightarrow \left(\trop|_{U(x)}\right)_*\left(c_1(L|_X,\lVert\cdot\rVert)^{\wedge d}|_{U(x)}\right)
    \end{equation}
    weakly against compactly supported continuous functions on $\trop(U(x))$. We now verify the properness of $\trop|_{\Sigma(U(x))}$. Let $K$ be a compact subset of $\trop(U(x))$. Then by (\ref{eq:construction of U p3}),
    \[
    (\trop|_{\Sigma(U(x))})^{-1}(K)=\Sigma(U(x))\cap \trop^{-1}(K)=\Sigma(V(x))\cap\trop^{-1}(\Omega)\cap\trop^{-1}(K).
    \]
    It follows from (\ref{eq:construction of U p2}) that $K\subset\Omega$ and thus
    \[
    (\trop|_{\Sigma(U(x))})^{-1}(K)=\Sigma(V(x))\cap\trop^{-1}(K),
    \]
    is compact as a closed subset of the compact space $\Sigma(V(x))$, which is the desired result. For the right-hand side of (\ref{eq:local description of MA on U}), recall that $C=\trop_*(\Sigma(V(x)))|_\Omega$ is the restriction to $\Omega$ of a balanced $(\Z,\Gamma)$-polyhedral complex in $N_\R$. Applying \Cref{Prop:MA poly weak convergence} to such an extension with $f=g_{\trop}$ and $f_i=g_i$, and then restricting to $\Omega$, we obtain
    \[
    \operatorname{MA}_{\operatorname{poly},C_x}(g_i)\longrightarrow\operatorname{MA}_{\operatorname{poly},C_x}(g_{\trop})
    \]
    weakly against compactly supported continuous functions on $|C_x|$. Combining this with the weak convergence in (\ref{eq:pf of main thm the properness of restriction}), we reduce to the case where $\lVert\cdot\rVert$ is a toric model metric.

    It follows from (\ref{eq:id of first Chern form on X cap T}) that
    \begin{equation}\label{eq:local thm U p1}
        c_1(L|_X,\lVert\cdot\rVert)^{\wedge d}|_{X^{\an}\cap T^{\an}}=c_1(L|_{X\cap T},\lVert\cdot\rVert)^{\wedge d}=\trop^\star(\ddc g_{\trop})^{\wedge d}.
    \end{equation}
    Recall that we have seen in the proof of \Cref{prop:MA support} that $\operatorname{supp}(\trop^\star(\ddc g_{\trop})^{\wedge d})\subset\Sigma(X^{\an}\cap T^{\an})$. The integration of a full-dimensional $\delta$-form over an analytic space is defined by integration over skeletons \cite[Definition 4.10]{DeltaLubinTate}, in our case,
    \begin{equation}\label{eq:local thm U p2}
        (\trop^{\star}(\ddc g_{\trop})^{\wedge d})|_{U(x)}=(\trop^\star(\ddc g_{\trop})^{\wedge d}\wedge\delta_{\mathcal{P}_x})|_{U(x)},
    \end{equation}
    where $\delta_{\mathcal{P}_x}$ is the integration current along the piecewise linear space $\mathcal{P}_x$ that is the chosen local chart of $\Sigma(X^{\an}\cap T^{\an})$ at $x$. 
    
    Let $\phi$ be a smooth function with compact support on $\trop(U(x))$, which is a $\delta$-form of codimension $0$ by \cite[Proposition 1.5]{DeltaLubinTate}. We are going to show that
    \begin{equation}\label{eq:local thm U WTS}
        \big( \trop|_{U(x)} \big)_*\left( c_1(L|_X,\lVert\cdot\rVert)^{\wedge d}|_{U(x)} \right)(\phi)= \left((\ddc g_{\trop})^{\wedge d}\wedge\delta_{C_x}\right)(\phi).
    \end{equation}
    Since $U(x)\cap\mathcal{P}_x=\Sigma(U(x))$, combining (\ref{eq:local thm U p1}) and (\ref{eq:local thm U p2}), we have that
    \begin{multline}\label{eq:local thm U WTS P1}
        \left(\trop|_{U(x)}\right)_*\left(c_1(L|_X,\lVert\cdot\rVert)^{\wedge d}|_{U(x)}\right)(\phi)\\
        = (\trop^\star(\ddc g_{\trop})^{\wedge d}\wedge\delta_{\mathcal{P}_x})|_{U(x)}( (\trop|_{U(x)})^*\phi )\\
        = \int_{\Sigma(U(x))}\trop^\star\left( (\ddc g_{\trop})^{\wedge d} \right)\wedge (\trop|_{U(x)})^*\phi\\
        = \int_{\Sigma(U(x))}\trop^\star\left( (\ddc g_{\trop})^{\wedge d}\wedge\phi \right),
    \end{multline}
    where the last equality can be obtained by a direct computation using (\ref{eq: pullback description}) together with \Cref{lem:restriction}.
    Since $\trop|_{\Sigma(U(x))}:\Sigma(U(x))\to\trop(U(x))$ is piecewise linear, it follows from \cite[Section 4.3(4)]{DeltaLubinTate} that
    \begin{equation}\label{eq:local thm U WTS P2}
        \int_{\Sigma(U(x))}\trop^\star\left( (\ddc g_{\trop})^{\wedge d}\wedge\phi \right)=\int_{\trop(U(x))} (\trop|_{\Sigma(U(x))})_* \left(\trop^\star\left( (\ddc g_{\trop})^{\wedge d}\wedge\phi \right)\right).
    \end{equation}
    This leads us to consider
    \[
    (\trop|_{\Sigma(U(x))})_* \left(\trop^\star\left( (\ddc g_{\trop})^{\wedge d}\wedge\phi \right)\right).
    \]
    Since $\trop|_{\Sigma(V(x))}$ has finite fibers as noted in \Cref{Section: tropical skeletons}, we can apply the pullback description of $\delta$-forms (\ref{eq: pullback description}) with $C=\Sigma(V(x))$ and $g=\trop|_{\Sigma(V(x))}$ to get that
    \[
    (\trop|_{\Sigma(V(x))})_*\left(\trop^\star\left( (\ddc g_{\trop})^{\wedge d}\wedge\phi \right)\right)= (\ddc g_{\trop})^{\wedge d}\wedge\phi\wedge\delta_{\trop_*(\Sigma(V(x)))}
    \]
    away from $\trop(\partial\, \Sigma(V(x)))$, where $\partial\, \Sigma(V(x))$ is the boundary of $\Sigma(V(x))$ in the tropical space $\Sigma(X^{\an}\cap T^{\an})$. Note that
    \[
    \partial\, \Sigma(V(x))\subset \Sigma(V(x))\backslash V(x)^\circ.
    \]
    To see this, take $y\in \Sigma(V(x))\cap V(x)^\circ$. Let $U'$ be an open subset of $(X\cap T)^{\an}$ such that $y\in U'\subset V(x)^\circ$. Then $U'\cap\Sigma(X^{\an}\cap T^{\an})$ is an open subset of $\Sigma(X^{\an}\cap T^{\an})$ which is contained in the polyhedral subset $\Sigma(V(x))$. Since $y\in U'\cap\Sigma(X^{\an}\cap T^{\an})$, we have $y\notin\partial\, \Sigma(V(x))$.
    
    It follows from (\ref{eq:construction of U p1}) that
    \[
    \Omega\cap\trop(\partial\, \Sigma(V(x)))=\varnothing.
    \]
    Therefore, together with (\ref{eq:construction of U p3}), we obtain that
    \begin{equation}\label{eq:local thm U WTS P3}
        (\trop|_{\Sigma(U(x))})_*\left(\trop^\star\left( (\ddc g_{\trop})^{\wedge d}\wedge\phi \right)\right)= (\ddc g_{\trop})^{\wedge d}\wedge\phi\wedge\delta_{\trop_*(\Sigma(U(x)))}.
    \end{equation}
    Combining (\ref{eq:local thm U WTS P1}) and (\ref{eq:local thm U WTS P3}), we get the desired identity (\ref{eq:local thm U WTS}).
\end{proof}

As an immediate consequence, we get \Cref{cor:intro maximal}. Recall that $\mathcal{P}_x$ can be any local $(\Z,\Gamma)$-polyhedral chart at $x$ compatible with $\trop|_{X^{\an}\cap T^{\an}}$.

\begin{cor}\label{cor: on maximal open faces}
    Let $x\in\Sigma(X^{\an}\cap T^{\an})$ and $\Delta$ be the open face of $\mathcal{P}_x$ such that $x\in\Delta$. Suppose that $\Delta$ has dimension $d=\operatorname{dim}(X)$. Then there exists an open neighborhood $U(x)$ of $x$ such that the following holds.

    For every toric line bundle $L$ on $Y$ and every semipositive toric metric $\lVert\cdot\rVert$ on $L$ with tropical Green function $g_{\trop}:N_\R\to\R$, we have
    \[
    c_1(L|_X,\lVert\cdot\rVert)^{\wedge d}|_{\Delta\cap U(x)}= d!\, m_\Delta\, \operatorname{MA}_\R(g_{\trop}|_{\trop(\Delta\cap U(x))}\circ\trop|_{\Delta\cap U(x)})
    \]
    as measures, where $m_\Delta$ is the weight of $\Delta$.
\end{cor}
\begin{proof}
    This follows immediately from \Cref{prop:TMA on maximal faces} and \Cref{thm:main thm}.
\end{proof}

On the right-hand side of (\ref{eq:local description of MA on faces and germs}), if $\lVert\cdot\rVert$ is a toric model metric associated with the convex $\Gamma$-rational piecewise affine tropical Green function $g_{\trop}$ and $\Delta$ has dimension $r<d$, then by taking $U(x)$ sufficiently small, depending on $\lVert\cdot\rVert$, we have that
\begin{equation}\label{eq:poly MA and real MA depending on metric}
    \operatorname{MA}_{\operatorname{poly},C_x}(g_{\trop})\big|_{\trop(\Delta\cap U(x))}=m\, \operatorname{MA}_\R(g_{\trop}|_{\trop(\Delta\cap U(x))})
\end{equation}
for a constant $m=m(\lVert\cdot\rVert)$ that depends on the chosen metric $\lVert\cdot\rVert$. 

Indeed, let $\widetilde{C}$ be the balanced weighted $(\Z,\Gamma)$-polyhedral complex inducing the polyhedral structure of $C_x$. Since $g_{\trop}$ is convex $\Gamma$-rational piecewise affine, we can take $\Pi_g$ to be the $(\Z,\Gamma)$-polyhedral complex in $N_\R$ given by the maximal domains of affinity of $g_{\trop}$. Taking a subdivision of $\widetilde{C}$ induced by $\Pi_g$, which is again a balanced $(\Z,\Gamma)$-polyhedral complex, we get a refinement $C_x'$ of the polyhedral structure of $C_x$ such that $g_{\trop}$ is affine on every face.

If $\trop(x)$ is not a vertex of $C_x'$, then, after shrinking $U(x)$, there does not exist any vertices of $C_x'$ in $\trop(\Delta\cap U(x))$ and thus
\[
\operatorname{MA}_{\operatorname{poly},C_x}(g_{\trop})\big|_{\trop(\Delta\cap U(x))}=\operatorname{MA}_\R(g_{\trop}|_{\trop(\Delta\cap U(x))})=0.
\]
If $\trop(x)$ is a vertex of $C_x'$, then, after shrinking $U(x)$, the point $\trop(x)$ becomes the unique vertex of $C_x'$ in $\trop(\Delta\cap U(x))$. It follows that
\[
\operatorname{MA}_{\operatorname{poly},C_x}(g_{\trop})\big|_{\trop(\Delta\cap U(x))}=m_1\, \delta_{\trop(x)}
\]
and
\[
\operatorname{MA}_\R(g_{\trop}|_{\trop(\Delta\cap U(x))})=m_2\, \delta_{\trop(x)}.
\]
As we have seen in the proof of \Cref{prop:ddf and MA on Rn}, the constant $m_2>0$ since it is the volume of the convex hull of the slopes of $g_{\trop}|_{\trop(\Delta\cap U(x))}$ at $\trop(x)$. Therefore, by taking $m=m_1/m_2$, we get (\ref{eq:poly MA and real MA depending on metric}).

However, following \Cref{tropical MA special class}, we can restrict our attention to a class of toric metrics for which the metric dependence is reduced to the explicit factor $\lambda^k$, while the remaining coefficient is a fixed geometric weight.

Let $L$ be a toric line bundle on $Y=Y_\Sigma$ with virtual support function $\psi:N_\R\to\R$. Recall that $\psi$ is an integral piecewise linear function which is linear on every cone of $\Sigma$. Refining $\mathcal{P}_x$ by the preimages under $\trop$ of the cones of $\Sigma$, we obtain a refinement, which we denote by $\mathcal{P}_x'$.

\begin{cor}\label{cor:special class}
    Let $x\in\Sigma(X^{\an}\cap T^{\an})$. Then there exists an open neighborhood $U(x)$ of $x$ such that the following holds.
    
    Let $L$ be a toric line bundle on $Y$ with virtual support function $\psi:N_\R\to\R$. Let $\mathcal{P}_x'$ be the refinement defined above. Let $\lVert\cdot\rVert$ be a semipositive toric metric on $L$ with tropical Green function $g_{\trop}:N_\R\to\R$ such that $\rho\coloneqq g_{\trop}+\lambda\psi$ is a smooth function on $N_\R$ for $\lambda\geq 0$. Let $\Delta$ be the open face of $\mathcal{P}_x'$ of dimension $d-k$ such that $x\in\Delta$. Then
    \[
    c_1(L|_X,\lVert\cdot\rVert)^{\wedge d}|_{\Delta\cap U(x)}=\frac{d!}{k!}\, \lambda^k\, m_{\Delta,x}\, \operatorname{MA}_\R(\rho|_{\trop(\Delta\cap U(x))}\circ\trop|_{\Delta\cap U(x)}),
    \]
    where
    \begin{itemize}
        \item with the face structure of $\Sigma(U(x))$ induced by the restriction of $\mathcal{P}_x'$, we have that $C_x=\trop_*(\Sigma(U(x)))$ is an open weighted representative of $\trop(X^{\an}\cap T^{\an},x)$ with support $\trop(U(x))$,
        \item $\trop(\Delta\cap U(x))$ is an open face of $C_x$ of dimension $d-k$,
        \item $m_{\Delta,x}$ is the weight of $\trop(\Delta\cap U(x))$ in the $(d-k)$-dimensional tropical cycle $(\operatorname{div}(-\psi))^k\cdot C_x$, and
        \item $\operatorname{MA}_\R(\rho|_{\trop(\Delta\cap U(x))}\circ\trop|_{\Delta\cap U(x)})$ is the $(d-k,d-k)$-real differential form whose coefficient is the determinant of the Hessian matrix of $\rho|_{\trop(\Delta\cap U(x))}\circ\trop|_{\Delta\cap U(x)}$.
    \end{itemize}
    For $\lambda=0$, the formula reduces to the case $k=0$, while the right-hand side vanishes for $k>0$.
\end{cor}
\begin{proof}
    Note that the restriction of $\mathcal{P}_x'$ to $\Sigma(U(x))$ induces a face structure on $\Sigma(U(x))$. It follows from \Cref{thm:main thm} that
    \[
    c_1(L|_X,\lVert\cdot\rVert)^{\wedge d}|_{\Delta\cap U(x)}=(\trop|_\Delta)^*(\operatorname{MA}_{\operatorname{poly},C_x}(g_{\trop})|_{\trop(\Delta\cap U(x))})
    \]
    for some open neighborhood $U(x)$ of $x$ and a weighted representative $C_x$ of $\trop(X^{\an}\cap T^{\an},x)$ with support $\trop(U(x))$. As in the proof of \Cref{thm:main thm}, we have that $C_x=\trop_*(\Sigma(U(x)))$, where the tropical skeleton $\Sigma(U(x))$ is a weighted abstract polyhedral complex with face structure obtained by restricting $\mathcal{P}_x$. 

    By the definition of $\mathcal{P}_x'$, the restriction of $\mathcal{P}_x'$ to $\Sigma(U(x))$ induces a polyhedral structure for which $\psi$ is affine on the tropicalization of every face. Hence, the induced polyhedral structure on $C_x=\trop_*(\Sigma(U(x)))$ is compatible with $\psi$. Since $\Delta$ is an open face of $\mathcal{P}_x'$ of dimension $d-k$ and $\trop|_\Delta$ is a linear isomorphism, the set $\trop(\Delta\cap U(x))$ is a $(d-k)$-dimensional open face of this polyhedral structure on $C_x$. 
    
    Let $m_{\Delta,x}$ be the weight of $\trop(\Delta\cap U(x))$ in the tropical cycle $(\operatorname{div}(-\psi))^k\cdot C_x$.

    By \Cref{prop: trop MA} and \Cref{tropical MA special class}, we have
    \[
    \operatorname{MA}_{\operatorname{poly},C_x}(g_{\trop})|_{\trop(\Delta\cap U(x))}=\frac{d!}{k!}\, \lambda^k\, m_{\Delta,x}\, \operatorname{MA}(\rho|_{\trop(\Delta\cap U(x))}).
    \]
    This proves the desired result.
\end{proof}

\begin{rmk}
    As a $\delta$-current supported on the tropical skeleton, a Monge-Amp{\`e}re measure for toric metrics on subvarieties of toric varieties can be computed through partitions of unity \cite[Definition 4.10(2)]{DeltaLubinTate}. The existence of a partition of unity is given by \cite[Proposition 2.5]{DeltaLubinTate}. On the other hand, \Cref{thm:main thm} and \Cref{cor: on maximal open faces} provide the required local structural description of the Monge-Amp{\`e}re measure.
\end{rmk}

{\footnotesize
\bibliographystyle{alpha}
\bibliography{ref.bib}
}

\end{document}